\documentclass{amsart}
 \usepackage[foot]{amsaddr}
\usepackage[utf8]{inputenc}
\usepackage[]{fontenc}
\usepackage[dvips]{epsfig}
\usepackage{hyperref}
\usepackage{amssymb}
\usepackage{amsthm}
\usepackage{graphicx}
\usepackage{import}
\usepackage{amsmath}
\usepackage{amscd}
\usepackage{mathtools}
\usepackage{verbatim}
\usepackage{tikz-cd}
\usepackage{color}
\usepackage{amsmath,tikz-cd}
\usepackage[all]{xy}
\usepackage[T1]{fontenc}

\usepackage{todonotes}

\allowdisplaybreaks

\newtheorem{thm}{Theorem}[section]
\newtheorem{prop}[thm]{Proposition}

\newtheorem{cor}[thm]{Corollary}

\theoremstyle{definition}
\newtheorem{definition}[thm]{Definition}

\theoremstyle{remark}
\newtheorem{remark}[thm]{Remark}

\numberwithin{equation}{section}

\renewcommand{\P}{\mathbb{P}}
\newcommand{\C}{\mathcal{C}}
\newcommand{\D}{\mathcal{D}}
\newcommand{\F}{\mathbb{F}}
\renewcommand{\H}{\mathcal{H}}
\newcommand{\Z}{\mathbb{Z}}
\newcommand{\Q}{\mathbb{Q}}

\begin{document}

%%
%% The title of the paper goes here.  Edit to your title.
%%

\title{Rational fibered Cubic fourfolds}

%%
%% Now edit the following to give your name and address:
%% 

\author{Hanine AWADA}
\email{hanine.awada@umontpellier.fr}
\address{Institut Montpellierain Alexander Grothendieck, CNRS, Universit\'e de Montpellier, Case Courrier 051, Place Eug\`ene Bataillon, 34095 Montpellier Cedex 5, France}

\urladdr{} % Delete if not wanted.

%%
%% If there is another author uncomment and edit the following.
%%

%\author{Second Author}
%\address{Department of Mathematics, University of South Carolina,
%Columbia, SC 29208}
%\email{second@math.sc.edu}
%\urladdr{www.math.sc.edu/$\sim$second}

%%
%% If there are three of more authors they are added in the obvious
%% way. 
%%

%%%
%%% The following is for the abstract.  The abstract is optional and
%%% if not used just delete, or comment out, the following.
%%%

\begin{abstract}
Some classes of cubic fourfolds are birational to fibrations over $\P^2$, where the fibers are rational surfaces. This is the case for cubics containing a plane (resp. an elliptic ruled surface), where the fibers are quadric surfaces (resp. del Pezzo sextic surfaces). It is known that the rationality of these cubic hypersurfaces is related to the rationality of these surfaces over the function field of $\P^2$ and to the existence of rational (multi)sections of the fibrations. 
We study, in the moduli space of cubic fourfolds,  the intersection of the divisor $\C_{8}$ (resp. $\C_{18}$) with  $\C_{14}$, $\C_{26}$ and $\C_{38}$, whose elements are known to be rational cubic fourfolds. We provide descriptions of the irreducible components of these intersections and give new explicit examples of rational cubics fibered in (quartic, quintic) del Pezzo surfaces or in quadric surfaces over $\P^2$. We also investigate the existence of rational sections for these fibrations. 
Under some mild assumptions on the singularities of the fibers, these properties can be translated in terms of Brauer classes on certain surfaces.
\end{abstract}

%%
%%  LaTeX will not make the title for the paper unless told to do so.
%%  This is done by uncommenting the following.
%%

\maketitle

%%
%% LaTeX can automatically make a table of contents.  This is done by
%% uncommenting the following:
%%

%\tableofcontents

%%
%%  To enter text is easy.  Just type it.  A blank line starts a new
%%  paragraph. 
%%

%%%%%%%%%%%%%%%%%%%%%%%%%%%%%%%%%%%%%%%%%%%%%%%%%%%%%%%%%%%%%%%%%%%%%%
\section{Introduction}\label{introduction}

A cubic fourfold is a smooth cubic hypersurface $X \subset \P^5$, that is the vanishing locus of a degree 3 homogeneous polynomial in 6 variables. Determining the rationality (or not) of $X$ is a very challenging open problem in algebraic geometry.

The fourfold $X$ is called \textit{special} if it contains an algebraic surface $S$ which is not homologous to a complete intersection. Hassett \cite{MR1738215} defines the loci $\C_d$ of special cubic fourfolds of discriminant $d$ and shows that these loci are non-empty irreducible divisors in the moduli space of cubic fourfolds $\mathcal{C}$, for $d > 6$ and $d \equiv 0,2 \pmod 6$. The values $d$ hence make up an infinite sequence of integers. Moreover, for an infinite, proper subset of the set of the divisors $\C_d$, the Hodge structure of the nonspecial cohomology of the cubic fourfold is essentially the Hodge structure of the primitive cohomology of a K3 surface. This K3 surface is said to be \textit{associated} to the special cubic fourfold. Hassett \cite[Section 5]{MR1738215} shows that special cubic fourfolds of discriminant $d$ have an associated K3 surface if and only if $d \nmid 4$, 9, or any odd prime $p \equiv 2 \pmod 3$. A natural suspicion is that any rational smooth cubic fourfold ought to have an associated K3 surface. For now, cubic fourfolds in $\C_{14}$, $\C_{26}$, $\C_{38}$ and $\C_{42}$ are proved to be rational (see \cite{c14}, \cite{MR3934590}, \cite{MR818549} or very recently \cite{2019arXiv190901263R}).\\

In this paper, we are interested in classes of cubic fourfolds that are birational to surface fibrations over $\P^2$. This is the case for cubics in $\C_8$ (resp. $\C_{18}$) containing a plane (resp. an elliptic ruled surface). The projection from the plane (resp. the linear system of quadrics through the elliptic ruled surface) gives the cubic a structure of quadric (resp. sextic del Pezzo) surface fibration over $\P^2$ (\cite{Ha2,2016arXiv160605321A}). The study of these fibrations and in particular the rationality of the surfaces in question over the function field of $\P^2$ is strongly related to the birational geometry of the fibered cubic fourfolds (see \cite{Ha2, 2016arXiv160605321A, MR3238111}). \\

In fact, Hassett \cite{Ha2} identifies countably many divisors in $\C_8$ consisting of cubic fourfolds containing a plane whose Clifford invariant is trivial, which implies rationality. We call Clifford invariant the Brauer class $\alpha \in Br(S)$ of the $\P^1$-bundle over a smooth degree two K3 surface $S$, which is the relative Hilbert scheme of lines $\H(0,1)$ in the fibers of the quadric fibration. This is well defined if the discriminant divisor of the quadric bundle is smooth \cite{MR3238111,Auel_2014}. The triviality of $\alpha$ is equivalent to the existence of a rational section of the associated quadric surface bundle  over $\P^2$. Each of these loci is a codimension two subvariety in the moduli space of cubic fourfolds $\C$. On the other hand, in \cite{2016arXiv160605321A}, the authors show other examples of rational cubic fourfolds in  $\C_{18}$, fibered in del Pezzo sextics, parametrized by a countably infinite union of codimension two subvarieties. A generic cubic fourfold $X$ in $\C_{18}$ contains an elliptic ruled surface $T$ of degree 6. The ideal of quadrics through $T$ define a rational map $X \dashrightarrow \P^2$ that displays the blow-up $\tilde{X}$ of $X$ along $T$ as a fibration in sextic del Pezzo surfaces over $\P^2$. Such a fibration is rational over the function field $\mathbb{C}(\P^2)$ when it has a rational section, and the codimension two loci described in \cite{2016arXiv160605321A} are exactly the loci were the associated del Pezzo fibration has a rational section. If the fibration has some mild \textit{``good''} proprieties (see  \textit{Def 3.1} - we will call these \it good del Pezzo fibrations \rm), one can associate two Brauer classes to such a fibration: the class  $\beta_2 \in Br(S)$ of a $\P^2$-bundle $\H(0,3)$ over a smooth degree two K3 surface $S$ and the class $\beta_3 \in Br(Z)$ of a $\P^1$-bundle $\H(0,2)$ over a triple cover $Z$ of $\P^2$. Here, $\H(0,3)$ and $\H(0,2)$ denote the relative Hilbert schemes over $\P^2$ parametrizing connected genus zero curves of respectively anticanonical degree three and two. The triviality of these Brauer classes $\beta_2$ or $\beta_3$ is equivalent to the existence of respectively a rational 2- or 3- multisection (see \cite{10.1093/imrn/rnz081}) of the del Pezzo fibration. The existence of a  rational section is equivalent to the triviality of both Brauer classes. \\

The goal of this paper is twofold. First, we study the geography of certain codimension two loci in the moduli space of cubic fourfold, obtained as intersections of $\C_8$ and $\C_{18}$ with other divisors that parametrize rational cubic fourfolds, $\C_{14}$, $\C_{26}$ and $\C_{38}$.  We give a complete description of the irreducible components of these intersections. 
These results are obtained by studying the intersection lattices of cubic fourfolds contained in these loci\footnote{In fact, for $\C_{14}\cap \C_8$ these are already known from \cite{MR3238111}}.
\begin{thm}
 \begin{enumerate}
     \item The intersection $\C_8 \cap \C_{26}$ in the moduli space of cubic fourfolds has eight irreducible components indexed by the discriminant $d \in \{29, 36, 48, 53, 61, 64, 68, 69\}$ of the intersection pairing of generic elements inside each component.
     \item The intersection $\C_8 \cap \C_{38}$ in the moduli space of cubic fourfolds has ten irreducible components indexed by the discriminant $d \in \{36, 45, 61, 68, 80, 85, 93,\\ 96, 100, 101\}$ of the intersection pairing of generic elements inside each component.
 \end{enumerate} 
\end{thm}

\begin{thm}\label{intersec}
\begin{enumerate}
\item The intersection $\C_{18}\cap \C_{14}$ has five irreducible components. Cubic fourfolds from these components have intersection lattice containing primitively embedded rank 3 sublattice of discriminant 84, 81, 72, 57 and 36.
\item The intersection $\C_{18}\cap \C_{26}$ has eight irreducible components. Cubic fourfolds from these components have intersection lattice containing a primitively embedded rank 3 sublattice of discriminant 48, 81, 108, 129, 144, 153, 156 and 36.
\item The intersection $\C_{18}\cap \C_{38}$ has eleven irreducible components. Cubic fourfolds from these components have intersection lattice containing a primitively embedded rank 3 sublattice of discriminant,  36, 81, 120, 153, 180, 201, 216, 225, 228, 45 and 57.
\end{enumerate}
\end{thm}
On the other hand, applying the results of the first part of the paper, we study the existence of rational (multi-) sections of the surface fibrations in cubic fourfolds, and study the birational geometry related to the existence of such cycles.

First we consider $\C_8$. Using our preceding intersection theoretical results, we give new examples of rational cubic fourfolds in $\C_{8}$ fibered in quadric surfaces with no rational section, hence not contained in the divisors of $\C_8$ described by Hassett.  We provide conditions on the intersection pairing of generic elements for the (non-) existence of the rational section (see \S \ref{section 4}).  The first example of such a cubic had been given in \cite{MR3238111}.

\begin{thm}
\begin{enumerate}
    \item Four of the components in the intersection $\C_8 \cap \C_{26}$ contain rational cubic fourfolds whose quadric fibration has no rational section.
    \item Five of the components in the intersection $\C_8 \cap \C_{38}$ in the moduli space of cubic fourfolds contain cubic fourfolds whose quadric fibration has no rational section.
\end{enumerate} 
\end{thm}

In order to translate this results in term of Brauer classes, we need the discriminant of the quadric fibration to be smooth. This is true for the generic cubic in $\C_8$ but since our examples lie in higher codimension, we need to construct at least a cubic with this property for each component.

Notably, with the help of Macaulay2, we produce an explicit rational example, with smooth discriminant divisor, with no rational section inside the intersection $\C_{8}\cap\C_{38}$.  Hence the Brauer class $\alpha$ on the degree two K3 surface is well defined and non-trivial.

For the case of $\C_{18}$, the scenario is far more complicated. Similarly to $\C_8$, we would like to consider examples of rational fibered cubic fourfolds such that the fibers are not rational over $\P^2$. Our starting idea was to produce an example of rational cubic fourfold fibrated in del Pezzo sextics, with no rational section over $\P^2$, using Macaulay2. But instead, we came across interesting examples of cubics fibered in quartic and quintic del Pezzo fibrations over $\P^2$. As far as we know, these are the first examples of this kind. 

\begin{thm}
    \begin{enumerate}
        \item The generic cubic fourfold in the locus of codimension 3 in $\C$, where the intersection lattice of cubics contains the following primitively embedded lattice is birational to a quintic del Pezzo fibration over $\P^2$:

$$\begin{pmatrix}
3 & 5 &1& 6\cr
5 & 13 & 0& \tau \cr
1 & 0&3&0\cr
6 & \tau &0& 18\cr 
\end{pmatrix}, $$
where $\tau \in \{9,10,11,12,13,14\}$.

        \item The generic cubic fourfold in a locus of codimension 4, where the intersection lattices of cubics contain the following primitively embedded lattice, is birational to a del Pezzo quartic fibration over $\P^2$. 

$$
\begin{pmatrix}
    3 & 5 &1&1& 6\cr
5 & 13 & 0&0& \tau \cr
1 & 0&3&0&0\cr
 1 & 0 & 0&3&0 \cr
 6 & \tau &0&0& 18\cr
\end{pmatrix},
$$
 with $\tau \in \{12,13,14\}$.  
    \end{enumerate}

\end{thm}

The existence of rational sections for these fibrations, and the related rationality and unirationality conditions over the function field of $\P^2$ are discussed in Section \ref{subsection 3.2} .

Unfortunately we do not manage to find explicit examples of cubic fourfolds fibrated in del Pezzo sextics, contained in the intersection of $\C_{18}$ with other divisors. Nevertheless, if we assume that these do exist in each component, our results easily allow to find the components where rational cubics with no rational sections should live.

\begin{thm}\label{nontrivial}
Let $X$ be a cubic fourfold, and let $\D_r$ denote the codimension 2 locus in $\C_{18}\cap\C_m$ ($m=14,\ 26,\ 38$), where generically cubics have rank 3 intersection matrix of discriminant $r$. Suppose $X$ lies in one of the following codimension 2 loci of $\mathcal{C}$ :

\begin{itemize}
\item inside $\C_{18}\cap \C_{14}$: $\D_{36}$ and $\D_{81}$;
\item inside $\C_{18}\cap \C_{26}$: $\D_{108}$ and $\D_{153}$;
\item inside $\C_{18}\cap \C_{38}$: $\D_{45}$, $\D_{81}$, $\D_{180}$, and $\D_{225}$.

\end{itemize}

Assume that the linear system $|2h - T|$ induces a sextic del Pezzo fibration over $\P^2$. Then the fibration has no rational section, and $X$ is rational.
\end{thm}

If furthermore the fibration is \textit{good}, then the Brauer class $\beta_2$ associated to the cubics of Thm. \ref{nontrivial} is non-trivial.

\medskip

The structure of this paper is as follows. In \S ~ \ref{section 2}, we give some notions and introduce some Hodge Theory and lattices associated to cubic fourfolds. In \S~ \ref{section 4}, we determine irreducible components of $\C_8 \cap\C_{26}$, $\C_8 \cap \C_{38}$ and study cubic fourfolds in these intersections giving new examples of rational cubic fourfolds fibered in quadric surfaces with a nontrivial Brauer class. In \S~ \ref{section 3}, we consider intersections of $\C_{18}$ with $\C_{14}$, $\C_{26}$ and $\C_{38}$ determining their irreducible components. We use these results to produce new examples of rational cubic fourfolds that are birational to fibrations in del Pezzo surfaces, and discuss their birational geometry.
\S~\ref{section 5} is the computational part of the paper, where the explicit examples that are necessary in the rest of the paper are constructed with Macaulay2. A code explaining the method to find other examples can be found attached.

\smallskip

%%%%%%%%%%%%%%%%%%%%%%%%%%%%%%%%%%%%%%%%%%%%%%%%%%%%%%%%%%%%%%%%%%%%%%%%
%% BEGIN The lyx specific LaTeX commands.
%%
\smallskip
\newblock{\textbf{Acknowledgments:} I am very grateful to the anonymous referee for clever suggestions, and for pointing out an error in a previous version of the paper. I also thank  Michele Bolognesi for stimulating discussions, exchanging ideas and suggesting many improvements.}

\section{Hodge theory for cubic fourfolds}\label{section 2}
Let X be a cubic fourfold over $\mathbb{C}$. Let us denote by $\C$ the moduli space of smooth cubic fourfolds in $\P^5$. It is a quasi-projective  twenty-dimensional variety. The Hodge diamond of X is as follows

\begin{center}

1\\0  \hspace{0.5cm} 0\\ 0 \hspace{0.5cm} 1 \hspace{0.5cm} 0\\ 0 \hspace{0.5cm} 0 \hspace{0.5cm} 0 \hspace{0.5cm} 0 \\ 0 \hspace{0.5cm} 1 \hspace{0.5cm} 21 \hspace{0.5cm} 1 \hspace{0.5cm} 0

\end{center}
 We focus on the middle cohomology of X, containing some nontrivial Hodge theoretic information.
 
 Let $L$ be the cohomology group $H^4(X,\Z)$, known as the \textit{cohomology lattice}, and $L_{prim} = H^4_{prim}(X,\Z) := \langle h^2 \rangle^{\perp}$ the \textit{primitive cohomology lattice}, where $h\in H^2(X,\Z)$ is the \textit{hyperplane class} defined by the embedding $X \subset \mathbb{P}^5$. Note that $L_{prim}$ is an even lattice (see \cite[\S 2]{MR1738215}).
 
 We consider more precisely the lattice of integral middle Hodge classes of $X$:
 $$M(X)= H^{2,2}(X) \cap H^4(X,\Z)= H^{2}(X, \Omega^2_X) \cap H^4(X,\Z)$$ equipped with the intersection form $(-,-)$. $M(X)$ is a positive definite lattice by the Hodge-Riemann bilinear relations. For cubic fourfolds, the (integral) Hodge conjecture holds (see  \cite[Theorem 18]{Voi} or \cite[Corollary 0.3]{mongardi2019curve} for a recent proof), and rational, algebraic and homological equivalences coincide for cycles of codimension 2. This means that the cycle map $CH^2(X) \to H^4(X,\mathbb{Z})$ is injective, where $CH^2(X)$ is the Chow group of codimension 2 cycles on $X$ up to rational equivalence (see \cite[\S 5]{colliot} or \cite{BP20}). In particular by the Hodge conjecture the algebraic cycles are the $(2,2)$-part of $H^4(X,\mathbb{Z})$. Denote by $d(M(X)) \in \Z$ the discriminant of the lattice $M(X)$, that is the determinant of the Gram matrix. The definition of \textit{special} cubic fourfold introduced in \S \ref{introduction} can be interpreted using $M(X)$: a smooth cubic fourfold X is \textit{special} if and only if the rank of $M(X)$ is at least 2. 

\begin{definition}
A \textit{labelling} of a special cubic fourfold consists of a positive definite rank two saturated sublattice $K_d$, with $h^2 \in K_d \subseteq M(X)$. The discriminant $d$ is the determinant of the intersection form on $K_d$.

\end{definition}
Special cubic fourfolds form a countably infinite union of irreducible divisors $\C_d$ in $\C$ with a labelling of discriminant $d$, where $d$ takes some integer values. 
  
Only few $\C_d$'s can be defined explicitly in terms of particular surfaces contained in X. For example, $\C_8$ is defined as the locus of cubic fourfolds containing a plane, $\C_{14}$ is the closure of the one containing a quartic scroll or equivalently a quintic del del Pezzo surface,  $\C_{18}$ is the closure of the locus of cubic fourfolds containing an elliptic ruled surface, and a few others.

\smallskip

We conclude this section with an important definition for the rest of the paper.

\begin{definition}
By a fibration in quadric (resp. sextic, quintic, quartic del Pezzo) surfaces over $\P^2$, we mean a morphism  $\sigma:Z \to \P^2$ from a fourfold $Z$, whose generic fiber is a smooth quadric (resp. sextic, quintic, quartic del Pezzo) surface.
\end{definition}

\section{\texorpdfstring{$\C_8$ and rational cubic fourfolds}{C8 and rational cubic fourfolds}} \label{section 4}

\subsection{Cubic fourfolds containing a plane}\label{subsection 4.1}

In this section, by lattice-theoretic calculations, we describe classes of rational cubic fourfolds containing a plane whose fibration in quadric surfaces has no rational section.

\medskip
Let Y be a cubic fourfold in $\C_8$ containing a plane $P$. For $h$ the hyperplane class of $Y$, the associated labelling $K_{8}$ has the following Gram matrix:
\begin{center}
    \bordermatrix{
 & h^2 & P \cr
h^2 & 3 & 1\cr
P & 1 & 3 
}.

\end{center}

Let $Q$ denote the class of the quadric surface residual to $P$ so that $h^2=P+Q$ and $\tilde{Y}$ the blow-up of Y along $P$. The projection from $P$ resolves into  a morphism
\begin{center}
    $q \colon \tilde{Y} \to \P^2 $.
\end{center}
The fibers of this morphism are quadric surfaces in the class $Q$ and cubic fourfolds containing a plane are birational to quadric surface bundles over $\mathbb{P}^2$ (see \cite{Ha2}).
Hassett \cite{Ha2} identifies countably many divisors in $\C_8$ parametrizing rational cubic fourfolds. Each of these loci is a codimension two subvariety in the moduli space of cubic fourfolds $\C$.
 
Recall that the \textit{discriminant divisor} $E$ is defined as the locus over which $q$ fails to be smooth. We say that the plane $P$ is \textit{good} if the  fibers of $q$ have at most isolated singularities. This is equivalent to having $E\subset \P^2$ smooth or also to having $X$ not containing another plane intersecting $P$ (see \cite[\S 1 Lemma 2]{MR2390291}, \cite[\S 1.5]{Auel_2014}). In this case, the double cover $S \to \P^2$  branched over $E$, a sextic curve, is a smooth K3 surface $S$ of degree 2. If $E$ is smooth, the relative Hilbert scheme of lines $\mathcal{H} (0,1)$ of the morphism $q$ is an \'{e}tale $\P^1$-bundle over $S$ (see \cite[\S 5]{BP20}). To such an object we can associate a Brauer class $\alpha \in Br(S)$ which is trivial if and only if $q$ has a rational section.

$$ \xymatrix @!0 @R=5mm @C=3cm {\relax
   \H(0,1)  \ar[r]^{\P^1}  \ar[rddddd] &  S \ar[ddddd]^{2:1}\\
  &\\
  &\\
    &\\
    &\\
    \tilde{Y} \ar[r]^{q} & \P^2 \supset E }$$
\begin{thm}[\cite{Ha2}]
If there is a class $C\in M(Y)$ 
such that $(C,Q)$ is odd then $Y$ is rational over $\mathbb{C}$.
\end{thm} 
In particular, $q$ has a rational section if and only if there exists an algebraic cycle $C \in M(Y)$ such that $(C,Q)=1$. In other words, the associated Brauer class $\alpha$ is trivial if and only if there exists an algebraic cycle $C \in M(Y)$ such that $(C,Q)=1$. According to Hassett, the triviality of the Brauer class implies the rationality of $Y$ over $\mathbb{C}$.\\

In the following, by lattice-theoretic calculations, we describe classes of rational cubic fourfolds containing a plane whose fibration in quadric surfaces has no rational section, respectively inside $\C_{8} \cap \C_{26}$ and inside $\C_{8} \cap \C_{38}$, hence not contained in the class of cubics described by Hassett.

\subsection{$\C_8 \cap \C_{26}$:} \label{C826}

A smooth cubic fourfold $Y$ is in $\C_8$ or $\C_{26}$ if and only if $M(Y)$ has primitive sublattice $K_8 :=\langle h^2, P \rangle$ or $K_{26}:=\langle h^2, S_{26} \rangle$  with the following Gram matrix
   \begin{center}
         \bordermatrix{
 & h^2 & S_{26} \cr
h^2 & 3 & 7\cr
S_{26} & 7 & 25
},
\end{center}
such that $P$ is a plane contained in  a general element of $\C_8$ and $S_{26}$ is a surface with one node obtained as the projection of a smooth del Pezzo surface $S \subset \P^7$ of degree seven from a line intersecting the secant variety of $S$ transversally (see \cite[\S 3]{MR3934590}), contained in  a general element of $\C_{26}$.

Thus a cubic fourfold $Y \in \C_8 \cap \C_{26}$ has a sublattice $\langle h^2, P, S_{26}\rangle \subset M(Y)$ with the following Gram matrix
\begin{center}
 \bordermatrix{
 & h^2 & P & S_{26} \cr
h^2 & 3 & 1 & 7 \cr
P & 1 & 3 & \tau \cr
S_{26} & 7 & \tau & 25\cr
},

\end{center}
for some $\tau=(P,S_{26})\in \Z$ depending on $Y$. The values of $\tau$ may be restricted following some properties and works of Voisin \cite{MR2390291} or Yang and Yu \cite{2019arXiv190501936Y}.

\smallskip
Denote by $M_{\tau}$ the lattice of rank 3 whose bilinear form has the previous Gram matrix. Let $\C_{M_{\tau}} \subset \C$ be the locus of smooth cubic fourfolds such that there is a primitive embedding $M_{\tau} \subset M(Y)$ of lattices preserving $h^2$.

\begin{prop}
The irreducible components of $\C_8 \cap \C_{26}$ are the subvarieties of codimension two $\C_{M_{\tau}}$ for $\tau \in \{-1,0,1,2,3,4,5,6\}$.
\end{prop}

\begin{proof}
We proceed as follows: first we find the set of possible values of $\tau$ for which $d(M_{\tau}) > 0$ and possibly nonempty. Then, for these values of $\tau$, we prove that $M_{\tau}$ is saturated. Finally, we find the associated irreducible components $\C_{M_{\tau}}$ which are nonempty i.e. don't have roots, that is vectors of norm 2.\\

By construction, $\C_8 \cap \C_{26} = \cup _{\tau \in \Z}\  \C_{M_{\tau}}$. We determine which components of $\C_{M_{\tau}}$ are possibly nonempty. Since for $Y \in \C_8 \cap \C_{26}$, $M(Y)$ is a positive definite lattice, by Sylvester's criterion, the sublattice $M_{\tau}$ must have positive discriminant. As $d(M_{\tau})=-3\tau^2+14\tau+53$, the only values of $\tau$ making a positive discriminant are $\tau \in I=\{-2,-1,0,1,2,3,4,5,6,7\} $. Hence, $\C_8 \cap \C_{26} = \cup _{\tau \in I}\ \C_{M_{\tau}}$.\\
Then, we prove that $\C_{M_{\tau}}$ is empty for $\tau = -2, 7$ by finding primitive roots (that is, primitive vectors of norm 2) in $M_{\tau,prim} = \langle h^2 \rangle^{\perp}$. Indeed, the vectors $(1,-3,0)$ and $(-3,2,1)$ form a basis for $M_{\tau,prim}$. For all $R_{a,b} \in M_{\tau,prim}$, $R_{a,b}= (a-3b,-3a+2b,b)$ for some $a,b \in \Z$, $\tau \in I$; for $\tau=-2$, we find primitive roots $\pm R_{0,1}$;  for $\tau=7$, we find primitive roots $\pm R_{1,1} = \pm (-2,-1,1)$. Hence, by \cite[\S 4 Proposition 1] {MR2390291}, $\C_{M_{\tau}}$ is empty for $\tau = -2,7$. We are left with $\C_{M_{\tau}}$ possibly nonempty only for $\tau \in \{-1,0,1,2,3,4,5,6\}$. The corresponding discriminants $d(M_\tau)$ are 36, 53, 64, 69, 68, 61, 48 and 29.\\

 \begin{figure}
  \begin{center}
 
\begin{tabular}{|c|c|c|c|c|c|c|c|c|}

\hline
$\tau$&-1& 0 & 1& 2 & 3 & 4 &5& 6 \\
\hline
$d(M_{\tau})$ &36& 53 & 64 & 69 & 68 & 61 &48& 29  \\
\hline
\end{tabular}
\caption{Irreducible components of $\C_{8} \cap \C_{26}$}
\label{t1}
   
  \end{center}
\end{figure}

 \smallskip
We prove now that the components $\C_{M_{\tau}}$ are irreducible. We first note that the rank of the lattice $M(X)$ is an upper-semicontinuous function on $\C$; the irreducible components of $\C_{M_{\tau}}$ correspond then to rank 3 overlattices B of $M_\tau$ (a finite index sublattice for some $\tau$ \textit{i.e.} $B/M_{\tau}$ is a finite abelian group) which is primitively embedded into $L$. 
By standard lattice theory, we have that for an embedding $M_{\tau} \hookrightarrow B$ with finite index $[B:M_{\tau}]= |B/M_{\tau}|$, $|d(B)|.[B:M_{\tau}]^{2} = |d(M_{\tau})|$ (see \cite{Nikulin_1980} or \cite{Ser});\\ 
 We will prove that no proper finite overlattices exist; for $\tau= 0,2,4,6$  the discriminants of $M_{\tau}$ are squarefree, so there are no proper finite overlattices.\\
For the remaining cases, we can take $h^2$ and $P$ as a part of a basis of the overlattice $B$. Let $U$ be a vector that completes this to a basis such that $U=x h^2+y P+z S_{26}$, with $x,\ y,\ z \in \mathbb{Q}$.  Consider the natural embedding of $M_{\tau}$ in B that can be written as follows:

\begin{center}
   $ \begin{pmatrix}
    1&0&x\\
    0&1&y\\
    0&0&z\\
    \end{pmatrix}^{-1}= \begin{pmatrix}
    1&0&-x/z\\
    0&1&-y/z\\
    0&0&1/z\\ 
    \end{pmatrix} \in \mathcal{M}_{3,3}(\Z)$.
\end{center}
We can take $z=\frac{1}{n}$, for some $n \in \Z$ and $x'=n x, y'=y n \in \Z$; then $U=\frac{1}{n}(x' h^2+y' P+S_{26})$. By adding multiples of $h^2$ and $P$, we may ensure that $0 \leq x',y'<n$.

 Computing intersections, we have the following:
 \begin{eqnarray*}
 (U,h^2) &=& \frac{1}{n}(3x'+y'+7)= a_0, \\
 (U,P)&=& \frac{1}{n}(x'+3y'+\tau) = b_0,\\
 (U,U) &= &\frac{1}{n^2}(3x'^2+3y'^2+14x'+2\tau y'+2x' y'+25)= c_0.
 \end{eqnarray*}

Hence, the Gram matrix of $B$ is:
 \begin{center}
          \bordermatrix{
& h^2 & P & U \cr
h^2 & 3 & 1 & a_0 \cr
P & 1 & 3 & b_0 \cr
U & a_0 & b_0 & c_0\cr
}.

 \end{center}

 Now we check each case separately for possible values of $\tau , n, x'$ and $y'$:
 \begin{enumerate}
     \item 
     $\tau =-1$: We see that $n$ can be 2, 3 or 6.\\ Remark that for all possible values of $n$, $x'$ and $y'$ other than ($n=3,\ x'=1$ and $y'=2$), the Gram matrix of $B \notin \mathcal{M}_{3,3}(\Z)$.  For ($n=3,\ x'=1$ and $y'=2$), $B$ has the following Gram matrix:
     \begin{equation*}
\begin{pmatrix}
3&1&4\\
1&3&2\\
4&2&6\\
\end{pmatrix},
\end{equation*}
which has the root $(-2,0,1)$. Then no such overlattices exist. Thus $\C_{M_{-1}}$ is irreducible.
     \item 
     $\tau =1$: We see that $n$ can be 2, 4 or 8.\\ Observe that for all possible values of $n$, $x'$ and $y'$, the Gram matrix of $B$ is not in $\mathcal{M}_{3,3}(\Z)$. Thus $\C_{M_1}$ is irreducible.
    
     \item
     $\tau =3$: $n$ can be 2.\\We notice that for all possible values of $x'$ and $y'$, the Gram matrix of $B$ is non-integral. Then no such overlattices exist. Thus $\C_{M_3}$ is irreducible.
     \item
     $\tau =5$: We observe that $n$ can be 2 or 4.\\
     Remark that for $n=4$, $|d(B)|=3$ which is impossible by \cite[Lemma 7.8]{yang2021lattice}.     
     Otherwise, for all possible values of $x'$ and $y'$, the Gram matrix of $B$ is non-integral. Then no such overlattices exist. Thus $\C_{M_5}$ is irreducible. \\

To check the (non)emptiness of these $\C_{M_{\tau}}$, we proceed as follows; for all $\tau \in \{-1,0,1,2,3,4,5,6\}$, $M_{\tau}$ is a positive definite saturated sublattice of rank 3: 
 \begin{center}
    $h^2 \in M_\tau \subset M(Y) \subset L$.
 \end{center}
 Furthermore, let $v = x h^2+ y P+ z S_{26} \in M_{\tau}$  for $x, y, z \in \Z$, we get
 \begin{center}
     $(v,v) = 3x^2 + 3y^2 + 25z^2 + 2xy+ 14xz + 2\tau yz$, $\tau \in \{-1,0,1,2,3,4,5,6\}$.
 \end{center}
  For these values of $\tau$, there exists no $v \in M_{\tau}$ such that $(v,v) = 2$. Thus, by \cite[Lemma 2.4]{2019arXiv190501936Y}, $\C_{M_{\tau}}$ is nonempty and has codimension two.\\  Note that the rank of the lattices is small compared to 21 so there is no issues with finding a primitive embedding in $L$.
   
 \end{enumerate}
 \end{proof} 
In the following, we address the question of the (non)triviality of the Brauer class.
\begin{thm}
Let $Y$ be a general cubic fourfold in $\C_8 \cap \C_{26}$ (so that $M(Y)$ has rank 3) containing a good  plane $P$. The Brauer class $\alpha \in Br(S)$ of $Y$ is trivial if and only if $\tau$ is even.
\end{thm}

\begin{proof}

 If $\tau$ is even, then a cycle $S_{26}+3P \in M(Y)$ exists such that $(S_{26}+3P,Q) =1-\tau \equiv 1\pmod 2$ (odd). Hence $\alpha$ is trivial by the application of the criterion of Hassett (see \cite{Ha2}).

\smallskip
 If $\tau$ is odd, then $M_{\tau}$ has rank 3 and even discriminant, hence $\alpha$ is nontrivial (see \cite[Proposition 2]{MR3238111}).
\end{proof}

\begin{cor}
The four irreducible components $\C_{M_{\tau}}$ of $\C_8 \cap \C_ {26}$ corresponding to $\tau= -1,$ $1$, $3$, $5$  contain examples of rational cubic fourfolds whose associated quadric surface bundles do not have a rational section.
\end{cor} 

\subsection{\texorpdfstring{$\C_8 \cap \C_{38}$}{C8 \textbackslash cap C38}}\label{subsection 4.3}

Using same methods as before, we compute the intersection between $\C_8$ and $\C_{38}$.

A cubic fourfold $Y$ is in $\C_8$ or $\C_{38}$ if and  only if $M(Y)$ has primitive sublattice $K_8 :=\langle h^2, P \rangle$ or $K_{38}:=\langle h^2, S_{38}\rangle$ with the following Gram matrix \begin{center}
\bordermatrix{
 & h^2 & S_{38} \cr
h^2 & 3 & 10\cr
S_{38} & 10 & 46
},
\end{center}such that $P$ is a plane and $S_{38}$ is the general degree 10 smooth surface of sectional genus 6 obtained as the image of $\mathbb{P}^2$ by the linear system of plane curves of degree 10 having 10 fixed triple points (see \cite[\S 4]{MR3934590}), contained in a general element of $\C_{38}$.\\
Hence, $Y \in \C_8 \cap \C_{38}$ has a sublattice $\langle h^2, P, S_{38}\rangle \subset M(Y)$ with Gram matrix: 

\begin{center}

\label{lattice2}
 \bordermatrix{
 & h^2 & P & S_{38} \cr
h^2 & 3 & 1 & 10\cr
P & 1 & 3 & \tau \cr
S_{38} & 10 & \tau & 46\cr
},
\end{center}
for some $\tau \in \Z$ depending on $Y$.
\begin{prop}
The irreducible components of $\C_8 \cap \C_{38}$ are the subvarieties of codimension two  $\C_{M_{\tau}}$ for $\tau \in \{-1,0,1,2,3,4,5,6,7,8\}$.
\end{prop}

\begin{proof}
As $d(M_{\tau})=-3\tau^2+20\tau+68$, the only values of $\tau$ inducing a positive discriminant are $\tau \in J=\{-2,-1,0,1,2,3,4,5,6,7,8,9\} $.\\
Let $M_{\tau,prim} = \{ (x,y,z) \in \mathbb{Z}^3\ \vert\   3x+y+10z=0\}$;
indeed, the vectors $(1,-3,0)$ and $(-3,-1,1)$ form a basis for $M_{\tau,prim}$; for all $R_{a,b} \in M_{\tau,prim}$, $R_{a,b}= (a-3b,-3a-b,b)$ for some $a,b \in \mathbb{Z}$, $\tau \in J$.\\
For $\tau=-2$, we find primitive  roots $\pm R_{-1,1}= \pm (-4,2,1)$. Hence, by  \cite[\S 4 Proposition 1]{MR2390291}, $\C_{M_{-2}}$ is empty ;
We are left with $\C_{M_{\tau}}$ possibly nonempty only for $\tau \in \{-1,0,1,2,3,4,5,6,7, 8,9\}$. The corresponding discriminants $d(M_\tau)$ are 45, 68, 85, 96, 101, 100, 93, 80, 61, 36 and 5.\\

Similarly to the previous section, we prove that the components $\C_{M_{\tau}}$ are irreducible. For $\tau =1,3,5,7$ the discriminants are squarefree, so there are no proper finite overlattices. For the remaining cases, we can prove the any overlattice in $\mathcal{M}_{3,3}(\mathbb{Z})$ has a primitive root. Therefore no such proper overlattices exist.

    \smallskip
Hence, for all $\tau \in \{-1,0,1,2,3,4,5,6,7,8,9\}$, $M_{\tau}$ is a positive definite saturated sublattice of rank 3: 
 \begin{center}
    $h^2 \in M_\tau \subset M(Y) \subset L$. 
 \end{center}
 Furthermore, let $v = x h^2+ y P+ z S_{38} \in M_{\tau}$, for $x, y, z \in \mathbb{Z}$, we get
 \begin{center}
     $(v,v) = 3x^2 + 3y^2 + 46z^2 + 2x y+ 20x z + 2\tau y z$;
 \end{center}
 for $\tau = 9$, we have that $(-2h^2-2P+S_{38},-2h^2-2P+S_{38})=2$. Then $\C_{M_{9}} \subset \C$ is empty (see \cite[Lemma 2.4]{2019arXiv190501936Y}). For the rest of the values of $\tau$, there exists no $v \in M_{\tau}$ such that $(v,v) = 2$, $\C_{M_{\tau}} \subset \C$ are nonempty irreducible and have codimension 2 (see \cite[Lemma 2.4]{2019arXiv190501936Y}). 
 \begin{figure}
\begin{center}
\begin{tabular}{|c|c|c|c|c|c|c|c|c|c|c|}

\hline
$\tau$ & -1&0 & 1 & 2 & 3 & 4 & 5 & 6 & 7&8 \\
\hline
$d(M_{\tau})$ & 45& 68 & 85 & 96 & 101 & 100 & 93 & 80 & 61&36  \\
\hline
\end{tabular}
\label{t2}

\end{center}
 \end{figure}

\end{proof}

\begin{thm}
Let $Y$ be a general cubic fourfold in $\C_8 \cap \C_{38}$ (so that $M(Y)$ has rank 3) containing a good plane $P$. The Brauer class $\alpha \in Br(S)$ of $Y$ is trivial if and only if $\tau$ is odd.
\end{thm}

\begin{proof}

 If $\tau$ is odd, then a  cycle $S_{38}+5P \in M(Y)$ exists such that $(S_{38}+5P,Q) =-\tau \equiv 1\ (mod$ $2)$ (odd).  Hence $\alpha$ is trivial by the application of the criterion (see \cite{Ha2}).

\smallskip
If $\tau$ is even, then $M_{\tau}$ has rank 3 and even discriminant,
hence $\alpha$ is nontrivial (see \cite[Proposition 2]{MR3238111}).
\end{proof}

\begin{cor}\label{cor 4.6}
$\C_8 \cap \C_ {38}$ has five smooth irreducible components, corresponding to $\tau = 0,2,4,6,8$, containing examples of rational cubic fourfolds whose associated quadric surface bundles do not have a rational section.
\end{cor}

In \S~ \ref{subsection 5.2}, we construct an explicit example of rational smooth cubic fourfold containing a \textit{good} plane in each irreducible component of the intersection $\C_{8} \cap \C_{38}$. This means that it makes sense to consider the associated Brauer class $\alpha \in Br(S)$ on the degree 2 K3 surface. In particular, in our explicit example the associated Brauer class is nontrivial since the quadric bundle has no rational section.

\section{\texorpdfstring{$\C_{18}$ and rational cubic fourfolds}{C18 and rational cubic fourfolds}}\label{section 3}

The goal of this section is twofold. First we describe the irreducible components of the intersection of $\C_{18}$ with the divisors $\C_{14}$,\ $\C_{26},$ and $\C_{38}$. Then we use these intersection theoretical results to showcase loci of \it rational \rm cubic fourfolds, which are fibered in del Pezzo surfaces. In each case, we consider gemetric conditions that imply the rationality of the fibers over the function field of $\P^2$ and hence of the associated cubic fourfolds.

 \subsection{Cubic fourfolds containing an elliptic ruled surface}\label{subsection 3.1}
A generic cubic fourfold in $\C_{18}$ contains an elliptic ruled surface of degree 6 (see \cite{2016arXiv160605321A}).\\ Let $X \in \C_{18}$  be a generic cubic fourfold containing an elliptic ruled surface $T$. The Gram matrix of the  associated labelling $K_ {18}$ is as follows: 
\begin{center}\bordermatrix{
 & h^2 & T \cr
h^2 & 3 & 6\cr
T & 6 & 18 
}.

\end{center}

To construct $T$, we start by fixing two disjoint planes $P'$ and $P$. Then, we choose 3 quadrics $\{Q_1,Q_2,Q_3\}$ containing the 2 disjoint planes. The intersection of the quadrics $Q_1 \cap Q_2 \cap Q_3$ has degree 8, and it is the union of $T$ and of the two planes $P$ and $P'$. \\
Let $\tilde{X}$ denote the blow-up of X along T and let 

\smallskip
\begin{center}
    $\pi \colon \tilde{X} \to \P^2 $
\end{center}
be the morphism induced by the linear system of quadrics containing $T$. For generic X, the generic fiber of $\pi$ is a del Pezzo surface of degree 6.

\medskip
Note that a singular del Pezzo surface is a surface \textbf{P} with \textit{ADE} singularities and ample anticanonical class such that $K_S^2=6$. They are classified as follows (see \cite[Section 2]{10.1093/imrn/rnz081} or \cite[Section 5]{2016arXiv160605321A} \textit{for a complete description}):
\begin{itemize}
    \item Type I: \textbf{P} has one $A_1$ singularity.
\item Type II: \textbf{P} has one $A_1$ singularity obtained in a different way from Type I.
\item Type III: \textbf{P} has two $A_1$ singularities.
\item Type IV: \textbf{P} has one $A_2$ singularity.
\item Type V: \textbf{P} has a $A_1$ and  a $A_2$ singularity.
\end{itemize}
\smallskip
Type I and II occur in codimension one in the moduli stack of sextic del Pezzo surfaces. Type III and IV occur in codimension two; type V occurs in codimension three.

\begin{definition}{\cite[Definition 11]{2016arXiv160605321A}} \textit{Let $\mathbb{P}$ be a smooth complex projective surface. A good sextic del Pezzo fibration consists of a smooth fourfold $Z$ and a flat projective morphism $\pi \colon Z \to \mathbb{P}^2$ with connected fibers such that the fibers are either smooth or singular sextic del Pezzo surfaces of Type I, II, III, IV or V. Let $B_i$ denote the closure of the locus of Type \textit{i} fibers in $\mathbb{P}^2$. $B_i$ has the following properties: \\
$\bullet$ $B_I$ is a non-singular curve;\\
$\bullet$ $B_{II}$ is a curve, non-singular away from $B_{IV}$;\\
$\bullet$ $B_{III}$ is finite and coincides with the intersection of $B_I$ and $B_{II}$, which is tranverse;\\
$\bullet$ $B_{IV}$ is finite and $B_{II}$ has cusps at $B_{IV}$;\\
$\bullet$ $B_{V}$ is empty.}
\end{definition} \label{gd}

The discriminant curve of a good sextic del Pezzo fibration $\pi$ 
has two irreducible components, a smooth
sextic curve $B_I$ and a sextic curve $B_{II}$ with 9 cusps
(see \cite{2016arXiv160605321A}, \cite{10.1093/imrn/rnz081} or \cite{BP20} \textit{for more details}). 
In this section, if a cubic $X\in \C_{18}$ has an associated good del Pezzo fibration, we will call it a \it good \rm cubic. More generally, for a \textit{good} del Pezzo fibration $\pi$ we consider the following construction.

Recall that a smooth sextic del Pezzo surface can be described as the blow up of $\P^2$ in three general points. Let us now consider two different Hilbert schemes of curves in the fibers of $\pi$. Let us denote by $\H (0, 3)$ the relative Hilbert scheme of connected genus zero curves with anticanonical degree 3 contained in the fibers. There are two 2-dimensional families of such curves on a smooth del Pezzo sextic. One is given by the strict transforms of the lines in $\P^2$ that do not pass through any of the three blown up points $p_1$, $p_2$ and $p_3 \in \P^2$ of the corresponding del Pezzo surface. The second one is given by conics passing through the three base points. Hence the Stein factorization of $\mathcal{H}(0,3) \to \mathbb{P}^2$ yields an \'{e}tale $\P^2$-bundle  $\H(0,3)$ over a smooth degree two K3 surface $S$ branched on $B_I$. 
 
 $$ \xymatrix @!0 @R=5mm @C=3cm {\relax
   \H(0,3)   \ar[r]^{\P^2}  \ar[rddddd] &  S \ar[ddddd]^{2:1}\\
  &\\
  &\\
    &\\
    &\\
    \tilde{X} \ar[r]^{\pi} & \P^2 \supset B_I }$$
We will denote by $\beta_2$ the Brauer class of this $\P^2$-bundle in $Br(S)$.

Let us consider now $\H(0,2) \to \P^2$ the relative Hilbert scheme of connected genus zero curves with anticanonical degree 2 contained in the fibers. The Stein factorization yields an \'{e}tale $\P^1$-bundle  $\H (0,2)$ over a non singular surface $Z$. In fact, associated to a \textit{good} del Pezzo fibration $\pi$, there is a non-singular triple cover $ Z \to \P^2$ branched along a cuspidal sextic $B_{II}$. On every geometric fiber, the $\P^1$-bundle is given by the strict transform of the lines through each of the 3 base points of the corresponding del Pezzo sextic.
 
  $$ \xymatrix @!0 @R=5mm @C=3cm {\relax
   \H(0,2)   \ar[r]^{\P^1}  \ar[rddddd] &  Z \ar[ddddd]^{3:1}\\
  &\\
  &\\
    &\\
    &\\
    \tilde{X} \ar[r]^{\pi} & \P^2 \supset B_{II} }$$
 
Similarly to the preceding case, we will denote by $\beta_3\in Br(Z)$ the Brauer class of this $\P^1$-bundle. 

 For more details on the Brauer classes of a sextic del Pezzo surface see \cite{Auel_2018} or \cite{Kuznetsov_2009}.

\smallskip

Furthermore, for a \textit{good} sextic del Pezzo fibration $\pi: \tilde{X}\to \P^2$, we have the following (see \cite[Proposition 5.20]{10.1093/imrn/rnz081}).

\begin{prop}\label{kuzsections}
\begin{enumerate}
    \item The \textit{Brauer class} $\beta_2 \in Br(S)$ is trivial if and only if $\pi$ has a rational 2-multisection.
    \item The \textit{Brauer class} $\beta_3 \in Br(Z)$ is trivial if and only if $\pi$ has a rational 3-multisection.
\end{enumerate}
\end{prop}

It is known \cite[Thm. 6]{2016arXiv160605321A} that the del Pezzo fibration associated to a good cubic fourfold $X\in \C_{18}$ always has a rational 3-section. It is worth to note that this rational 3-section does not come as the strict transform of a 2-dimensional algebraic cycle in $X$ but rather as a surface obtained as the inverse image of a curve inside the elliptic ruled surface. This is surely of interest, but our analysis concentrates on 2-cycles on cubic fourfolds, since are the ones we can control better on special loci of the moduli space. This is why in the following we will concentrate on rational (multi)sections that are obtained from 2-cycles on the cubic fourfold. 
In this sense, it is worth recalling the following results (\cite[Prop. 2.3]{ouc21} and \cite[Prop. 8]{2016arXiv160605321A}).

\smallskip

Let $F\in M(X)$ denote the class of the fiber of the rational map given by quadrics through $T$, for a cubic fourfold $X\in \C_{18}$. That is,

$$F=4h^2 - T.$$

\begin{prop}\label{ouch}
 The del Pezzo fibration $\pi: \tilde{X} \to \P^2$ has a rational
section if and only if there is an algebraic cycle class $R \in H^{2,2}(X, \Z)$ such that the intersection $(R,F)$ is 1 or 2.
\end{prop}

\begin{prop}\label{equirat}
Let $\mathcal{G} \to \P^2$ be a del Pezzo sextic fibration. The following are equivalent:
\begin{enumerate}
\item $\mathcal{G}$ is rational over the function field $\mathbb{C}(\P^2)$;\\
\item $\mathcal{G}$ admits an rational section $\P^2 \dashrightarrow \mathcal{G}$;\\
\item $\mathcal{G}$ admits a rational multi-section of degree prime to six;\\
\item if $\mathcal{G}$ is good, the Brauer classes $\beta_2$ and $\beta_3$ are trivial.
\end{enumerate}
\end{prop}

In the next section we will describe the interplay between special surfaces contained in cubic fourfolds and irreducible components of $\C_{18} \cap \C_{14}$. 

\subsection{\texorpdfstring{$\C_{18} \cap \C_{14}$}{C18 \textbackslash cap C14}}\label{subsection 3.2}
Let $X$ be a cubic fourfold in $\C_{18} \cap \C_{14}$; $M(X)$ has primitive sublattices $K_{18} :=\langle h^2, T \rangle$ and $K_{14}:=\langle h^2, D \rangle$ with its Gram matrix \begin{center}
 \bordermatrix{
 & h^2 & D \cr
h^2 & 3 & 5\cr
D & 5 & 13 
},
\end{center} such that $D$ is (possibly  a degeneration of) a quintic del Pezzo surface (see \cite{c14}), $T$ is  (possibly a degeneration of) an elliptic ruled surface. 

Hence, $X \in \C_{18} \cap \C_{14}$ has a sublattice $\langle h^2, D,T \rangle \subset M(X)$ with the following Gram matrix

\begin{center}
 \bordermatrix{
 & h^2 & D & T\cr
h^2 & 3 & 5 & 6\cr
D & 5 & 13 & \tau \cr
T & 6 & \tau & 18\cr},

\end{center}
for some $\tau \in \Z$ depending on $X$.

\smallskip
Denote by $M_{\tau}$ the lattice of rank 3 whose bilinear form has the previous Gram matrix (by abuse of notation, we use $M_{\tau}$ to denote the Gram matrix too) We will denote by $\C_{M_{\tau}} \subset \C$ the locus of cubic fourfolds such that there is a primitive embedding $M_{\tau} \subset M(X)$ of lattices preserving $h^2$.

\begin{thm} \label{C18/14}

The irreducible components of $\C_{18} \cap \C_{14}$ are the subvarieties of codimension two $\C_{M_{\tau}}$ given by the following rank 3 Gram matrix

  \begin{center}
       $M_{\tau} :=$\bordermatrix{
 & h^2 & D & T\cr
h^2 & 3 & 5 & 6\cr
D & 5 & 13 & \tau \cr
T & 6 & \tau & 18\cr
},
     \end{center}
     where  $10 \leq  \tau \leq 14$.
Moreover, the associated discriminants $d(M_{10}),\dots, d(M_{14})$ are 84, 81, 72, 57, 36 respectively.  

\end{thm}

\begin{proof}
We proceed as follows: first we find the set of possible values of $\tau$ for which $d(M_{\tau}) > 0$. Then, for these values of $\tau$, we prove that $M_{\tau}$ is saturated. Finally, we find the associated irreducible components $\C_{M_{\tau}}$ which are nonempty i.e. don't have roots, that is vectors of norm 2.  \\Note that $M_{\tau}$, defined as a sublattice of a positive definite lattice, must have positive discriminant by Sylvester's criterion. As $d(M_{\tau})=-3\tau^2+60\tau-216$, the only values of $\tau$ making a positive discriminant are $\tau \in \{5,6,7,8,9,10,11,12,13,14,15\} $.\\For these values of $\tau$, $M_{\tau}$ is a positive definite sublattice of rank 3 containing $h^2$. The corresponding discriminants $d(M_\tau)$ are 9, 36, 57, 72, 81, 84, 81, 72, 57, 36, 9. Note that there are isometries between the lattices $M_{5}$, $M_{6}$, $M_{7}$, $M_{8}, M_{9}$, and  $M_{15}$, $M_{14}$, $M_{13}, M_{12}$, $M_{11}$ respectively (see \cite[Remark 7.7]{yang2021lattice}). Thus, next, we may only consider $M_{\tau}$ for $\tau= \ 10,\ 11,\ 12,\ 13,\ 14,\ 15$.

We check whether  $\C_{M_{\tau}}$ is irreducible. 
For $\tau =5$, the discriminant $d(M_{5})$ is squarefree, so there are no proper finite overlattices of $M_{5}$.
For the remaining cases, we can take $h^2$ and $D$ as a part of a basis of an overlattice $B$ of $M_{\tau}$ and complete it with $V$ such that $V=x h^2 +y D +zT$ with $x,y,z \in \Q$.\\
We can take $z=\frac{1}{n}$, for some $n \in \Z$ and $x'=n x, y'=y n \in \Z$; then $V=\frac{1}{n}(x' h^2+y' D+T)$. By adding multiples of $h^2$ and $D$, we may ensure that $0 \leq x',y'<n$.\\ Note that $n=[B:M_{\tau}]=[B_{prim}:M_{\tau, prim}]$, with  $B_{prim}$ the finite proper overlattice of $M_{\tau,prim}$ (this follows from standard lattice theory).

 Computing intersections:
 \begin{eqnarray*}
(V,h^2) &=&\frac{1}{n}(3x'+5y'+6)= a_1, \\
(V,D) &=& \frac{1}{n}(5x'+13y'+\tau) = b_1,\\
(V,V) &=& \frac{1}{n^2}(3(x')^2+13(y')^2+12 x'+2\tau y'+10x' y'+18)= c_1.
 \end{eqnarray*}
Then $B$ has the following Gram matrix:
 \begin{center}
                            \bordermatrix{
& h^2 & D & V\cr
h^2 & 3 & 5 & a_1 \cr
D & 5 & 13 & b_1\cr
V& a_1 & b_1 & c_1\cr
}.

 \end{center}

 Now we check each case separately for possible values of $\tau , n, x'$ and $y'$.\\
 For ($\tau=10,\ n=2,\ x'=1,\ y'=1$), $B$ has the following Gram matrix
 \begin{center}
          $\begin{pmatrix} 3&5 & 7\\5&13&14\\7&14 & 19  \end{pmatrix}$, 
     \end{center}
    of discriminant $21$.
         The vectors $(-5,3,0)$ and $(-7,0,3)$ form a basis for $B_{prim}$ which has the following Gram matrix:
      \begin{center}
          $\begin{pmatrix} 42 & 21\\21&33  \end{pmatrix}$.
     \end{center}
   Since $B_{prim}$ is not even, no such overlattices exist and this component is irreducible.\\
    No overlattices exist in the following cases.\\
    For ($\tau=11,\ n=3,\ x'=2,\ y'=0$), $B$ has the following Gram matrix
  \begin{center}
          $\begin{pmatrix} 3&5 & 4\\5&13&7\\4&7 & 6  \end{pmatrix}$, 
     \end{center}
    which has a root $ (-3,1,1)$. \\
 For ($\tau=12,\ n=2,\ x'=1,\ y'=1$), $B$ has the following Gram matrix
  \begin{center}
          $\begin{pmatrix} 3&5 & 7\\5&13&15\\7&15 & 20  \end{pmatrix}$, 
     \end{center}
    which has roots $\pm (-1,-1,1)$.\\
     For ($\tau=12,\ n=3,\ x'=0,\ y'=0$), $B$ has the following Gram matrix
  \begin{center}
          $\begin{pmatrix} 3&5 & 2\\5&13&4\\2&4 & 2  \end{pmatrix}$, 
     \end{center}
    which has roots $\pm (0,0,1)$.\\
     For ($\tau=12,\ n=6,\ x'=3,\ y'=3$), $B$ has the following Gram matrix
  \begin{center}
          $\begin{pmatrix} 3&5 & 5\\5&13&11\\5&11 & 10  \end{pmatrix}$, 
     \end{center}
    which has a root $ (-1,-1,2)$.\\
     For ($\tau=14,\ n=2,\ x'=1,\ y'=1$), the Gram matrix of $B$ is:
  \begin{center}
          $\begin{pmatrix} 3&5 & 7\\5&13&16\\7&16 & 21  \end{pmatrix}$.
     \end{center}
    $B$ has roots such as $(0,-1,1)$. \\
      For ($\tau=14,\ n=3,\ x'=2,\ y'=0$), the Gram matrix of $B$ is:
  \begin{center}
          $\begin{pmatrix} 3&5 & 4\\5&13&8\\4&8 & 6  \end{pmatrix}$.
     \end{center}
    $B$ has roots such as $(1,1,-2)$. \\
      For ($\tau=14,\ n=6,\ x'=5,\ y'=3$), the Gram matrix of $B$ is:
  \begin{center}
          $\begin{pmatrix} 3&5 & 6\\5&13&13\\6&13 & 14  \end{pmatrix}$.
     \end{center}
    $B$ has roots such as $(-2,0,1)$. \\
     For ($\tau=15,\ n=3,\ x'=0,\ y'=0$), $B$ has the following form
  \begin{center}
          $\begin{pmatrix} 3&5 & 2\\5&13&5\\2&5 & 2  \end{pmatrix}$, 
     \end{center}
    which has roots $\pm (0,0,1)$. \\

For all other possible values $\tau , n, x'$ and $y'$, the Gram matrices of $B$ are not in $\mathcal{M}_{3,3}(\Z)$.

The lattice $M_{\tau}$ is saturated definite positive of rank 3 such that: 
 \begin{center}
    $h^2 \in M_\tau \subset M(X) \subset L$. 
 \end{center}
Let $v = x h^2+ y S_{14}+ z T \in M_{\tau}$. For $x, y, z \in \Z$, we get
 \begin{center}
     $(v,v) = 3x^2 + 10y^2 + 18z^2 + 8x y + 12x z +2\tau y z$.
 \end{center}
 We prove that $\C_{M_{15}}$ is empty by finding roots. Indeed, $(-h^2 - D + T,- h^2 -D + T)=2$;  for the remaining values of $\tau$,
 $M_{\tau}$ has no vectors $v$ such that $(v,v) = 2$. Hence, $\C_{M_{\tau}} \subset \C$ is irreducible nonempty (see \cite[Lemma 2.4]{2019arXiv190501936Y}). 
 
\end{proof}

In the rest of this section we will consider cubics living in (at least) codimension 1 loci inside $\C_{18}$, and hence the existence of the sextic del Pezzo fibration is not guaranteed by \cite{2016arXiv160605321A}. For some of these irreducible loci we will find explicit examples of cubics which are birational to fibrations in  (quartic, quintic) del Pezzo surfaces over $\P^2$. This implies that the generic cubic has the same feature. For the loci where we do not show explicit examples, we tacitly assume that the generic cubic is birational to a sextic del Pezzo fibration.

\begin{definition}
Let $F\in M(X)$ be the class $4h^2-T$. We will call a \it 1-sectional (resp. 2 or 3-sectional) cycle \rm a dimension 2 algebraic cycle $W\in M(X)$ such that we have $(W,F)=1$ (resp. 2 or 3).
\end{definition}
      
Let us assume that the class $F=4h^2-T \in M(X)$ gives a two-dimensional linear system of sextic del Pezzo surfaces. Under this assumption, we now check the existence of 1,2 and 3-sectional cycles in $M(X)$ for the corresponding del Pezzo fibration on each component of $\C_{18} \cap \C_{14}$. Let $W_{a,b,c} \in M(X)$ be a cycle such that
\begin{center}
    $W_{a,b,c} := a h^2 + bD +c T$, for $a,b,c \in \Z$,
\end{center}
with $D$ a (degeneration of) quintic del Pezzo surface  and  $T$ an elliptic ruled surface (or a degeneration of such surface).\\
We have that $(W_{a,b,c},F) = 6a + (20-\tau)b +6c$. We will now check, for the possible values of $\tau$, the existence of a 1, 2 or 3-sectional cycles.

\smallskip

In the following table, we collect our results about the existence of sectional cycles for the different components of $\C_{18} \cap \C_{14}$. The symbol $\emptyset$ means that there are no sectional cycles of that given type.

\medskip
\begin{center}
\begin{tabular}{|c|c|c|c|}\hline
$\tau $&$(W_{a,b,c},F)=1$ &$(W_{a,b,c},F)=2$ & $(W_{a,b,c},F)=3$\\
\hline
 10 &$\emptyset$& $W_{0,2,-3}$&$\emptyset$\\
\hline
11&$\emptyset$&$\emptyset$&$W_{0,1,-1}$\\
\hline
12&$\emptyset$&$W_{0,1,-1}$&$\emptyset$\\\hline
13&$W_{0,1,-1}$&$W_{0,2,-2}$&$W_{0,3,-3}$\\\hline
14&$\emptyset$& $\emptyset$& $\emptyset$\\\hline

\end{tabular}
\end{center}

\begin{remark}\label{remarks}

\hspace{1cm}

\begin{enumerate}

\item We observe that, by Prop. \ref{ouch}, the del Pezzo fibrations for $\tau=11,\ 14$ have no rational section, hence - by Prop. \ref{equirat} - $\tilde{X}$ in these cases is not rational on $\mathbb{C}(\P^2)$ but nevertheless rational since the cubic is contained in $\C_{14}$. \\

\item Again by Prop. \ref{equirat}, this means that, if one of these cubics is good, then $\beta_2$ and $\beta_3$ are not both trivial for cubics in these components. Since by \cite[Thm. 6]{2016arXiv160605321A} there is always a rational 3-section of the corresponding del Pezzo fibration, then by Prop. \ref{kuzsections} the class $\beta_3$ is always trivial. This means that the class $\beta_2$ is nontrivial for good cubics in this component.\\

\item Of course, the existence of a 1-sectional cycle is the strongest condition, and implies the existence of 2 and 3-sectional cycles.\\

\item For the components with $\tau=10$ and 12 it is worth observing that, by Prop. \ref{ouch}, the associated del Pezzo sextic fibrations have a rational section, though this does not come from a 1-sectional cycle in the cubic. It is nevertheless not hard to construct a rational section of the del Pezzo fibration, starting from the strict transform of a 2-sectional cycle, in the spirit of the proof of \cite[Prop. 2.3]{ouc21}.

\end{enumerate}

\end{remark}

The upshot  is that for cubics in the component obtained with $\tau=11, 14$, if they are birational to a fibration in del Pezzo sextics over $\P^2$, the rationality of the del Pezzo surfaces over the function field of $\P^2$ is not necessary for rationality, and these examples are not contained in the codimension 2 loci of rational cubics described in \cite{2016arXiv160605321A}. 

All the explicit examples of cubics $X$ in the intersection $\C_{18} \cap \C_{14}$ that we were able to construct with Macaulay2 contain not only $T$, but also one of or both the disjoint planes $P$ and $P'$ that are residual to $T$ in the complete intersection of the three quadrics (see Subsection \ref{subsection 3.1}). The consequence is that they are indeed birational to a del Pezzo fibration, but of lower degree. We will now give a theoretical description of these loci, where 
our explicit examples live.

\smallskip

In particular, let us consider the sublocus of $ \C_{18} \cap \C_{14}$ parametrizing cubics containing the elliptic ruled surface $T$, $P$, $P'$ and a del Pezzo surface that intersects both $P$ and $P'$ along a smooth conic. In fact, a cubic $X$ containing $T$ and $D$ automatically contains $P$ and $P'$, since $X\cap P$ (resp. $X\cap P'$) contains a plane cubic and a conic, hence $P\subset X$ (resp. $P'\subset X$). 
This locus has codimension 4 in the moduli space, and its generic element has $M(X)$ of rank 5. Such cubics can be found following instructions in the attached file using Macaulay2. Moreover, an explicit example is given in Section 5. 
Recall that the two planes intersect $T$ along a cubic curve.

\begin{prop}
Let $P$ and $P'$ be two disjoint planes in $\P^5$, $T$ the elliptic ruled surfaces obtained from $P$ and $P'$, and $D$ a del Pezzo surface intersecting each of the planes in a smooth conic.
    Let $X \in \C_{18} \cap \C_{14}$ be a cubic containing $T,\ P,\ P'$ and $D$. Then $X$ has a sublattice $M_{\tau}:=\langle h^2,D, P,P',T \rangle \subset M(X)$ of rank 5 whose bilinear form has the following Gram matrix: \begin{center}
 \bordermatrix{
 & h^2 & D & P&P'& T\cr
h^2 & 3 & 5 &1&1& 6\cr
D & 5 & 13 & 0&0& \tau \cr
P & 1 & 0&3&0&0\cr
P' & 1 & 0 & 0&3&0 \cr
T & 6 & \tau &0&0& 18\cr},

\end{center}
for some $\tau=(D,T) \in \Z$ depending on $X$.
\end{prop}
\begin{proof}
    
     Let $S_1$, $S_2$ be two smooth surfaces in a smooth cubic hypersurface such that the scheme-theoretic
intersection $S_1 \cap S_2$ contains a smooth curve $C$ of degree $d$ and genus $g$. Then the multiplicity of intersection
of $S_1$ and $S_2$ along $C$ is given (see \cite{c14}) by:
$$mult _C(S_1,S_2) = 3d + K_{S_1}.C + K_{S_2}.C + 2- 2g$$
where $K_{S_i}$
denotes the canonical class of $S_i$.\\
In our case, we have that $K_P=K_{P'}=-3H$ and $K_D=-H$, with $H$ the hyperplane class. Since the quintic del Pezzo surface $D$ and the two disjoint planes intersect in a smooth conic curve $C$ of genus 0, we have the following:$$mult _C(D,P)=mult _C(D,P') = 3.2 -3.2-2+2=0.$$
We do the same in order to compute the intersection $(T,P)=(T,P')$. Since $T$ and the plane intersect in a curve $C'$ of degree 3 and genus 1. The formula gives us:$$mult _{C'}(T,P)=mult _{C'}(T,P') = 3.3 -3.3+2-2=0.$$

\end{proof}

\smallskip
We will denote by $\C_{M_{\tau}} \subset \C$ the locus of cubic fourfolds such that there is a primitive embedding $M_{\tau} \subset M(X)$ of lattices preserving $h^2$.

\begin{thm} \label{C18/14'}
The locus of smooth cubic fourfolds associated to the the following rank 5 Gram matrix 

\begin{center}
 \bordermatrix{
 & h^2 & D & P&P'& T\cr
h^2 & 3 & 5 &1&1& 6\cr
D & 5 & 13 & 0&0& \tau \cr
P & 1 & 0&3&0&0\cr
P' & 1 & 0 & 0&3&0 \cr
T & 6 & \tau &0&0& 18\cr}\end{center}
 with $\tau=(D,T)$, has exactly 3 irreducible components $\C_{M_{12}}$, $\C_{M_{13}}$ and $\C_{M_{14}}$.  Moreover, the associated discriminants $d(M_{12}), \ d(M_{13})$ and $d(M_{14})$ are 108, 123 and 96 respectively.  

\end{thm}

\begin{proof}

Since  $M_{\tau}$ is defined as a sublattice of a positive definite lattice, it must have positive discriminant by Sylvester's criterion. We find the set of possible values of $\tau$ for which $d(M_{\tau}) > 0$. 

As $d(M_{\tau})=-21\tau^2+540\tau-3348$, the only values of $\tau$ making a positive discriminant are $\tau \in \{11,\ 12,\ 13,\ 14,\ 15\} $. For these values of $\tau$, $M_{\tau}$ is a positive definite sublattice of rank 5 containing $h^2$ with discriminants $d(M_\tau)$ equal to 51, 108, 123, 96, 27 respectively.\\
We now prove that $M_{\tau}$ is saturated for the possible values of $\tau$. 

We will prove that no proper finite overlattices exist; for $\tau =11$ the discriminant is squarefree, so there are no proper finite overlattices.
For the remaining cases, we can take $h^2$, $D$, $P$ and $P'$ as a part of a basis of an overlattice $B$ and complete it with $V$ such that $V=x h^2 +y D +zP+tP'+uT$ with $x,y,z,t,u \in \Q$. Consider the natural embedding of $M_{\tau}$ in B that can be written as follows:

\begin{center}
   $ \begin{pmatrix}
    1&0&0&0&x\\
    0&1&0&0&y\\
    0&0&1&0&z\\
    0&0&0&1&t\\
    0&0&0&0&u\\
    \end{pmatrix}^{-1}= \begin{pmatrix}
    1&0&0&0&-x/u\\
    0&1&0&0&-y/u\\
    0&0&1&0&-z/u\\ 
    0&0&0&1&-t/u\\ 
    0&0&0&0&1/u\\ 
    \end{pmatrix} \in \mathcal{M}_{5,5}(\Z)$.
\end{center}
We can take $u=\frac{1}{n}$, for some $n \in \Z$ and $x'=n x,\ y'=y n,\ z'=z n, \ t'=t n \in \Z$; then $V=\frac{1}{n}(x' h^2+y' D+ z'P+t'P'+T)$. By adding multiples of $h^2$, $D$, $P$ and $P'$, we may ensure that $0 \leq x',y',z',t'<n$. 
 Computing intersections, we have the following:
 \begin{eqnarray*}
(V,h^2) &=&\frac{1}{n}(3x'+y'+z'+t'+6)= a, \\
(V,D) &=& \frac{1}{n}(5x'+13y'+\tau) = b,\\
(V,P) &=& \frac{1}{n}(x'+3z') = c,\\
(V,P') &=& \frac{1}{n}(x'+3t') = d,\\
(V,V) &=& \frac{1}{n^2}(3(x')^2+13(y')^2+3(z')^2+3(t')^2+10x'y'+2y'z'+2x't'+12 x'\\&&+2\tau y'+18)= e.
 \end{eqnarray*}
Then the Gram matrix of $B$ is:
 \begin{center}
\bordermatrix{
& h^2 & D &P&P'& V\cr
h^2 & 3 & 5 & 1&1&a \cr
D & 5 & 13 &0&0&b \cr
P & 1 & 0 & 3&0&c\cr
P' & 1 & 0&0&3 & d\cr
V& a & b &c&d& e\cr
}.

 \end{center}

 Now we check each case separately for possible values of $\tau , n, x'$, $y' \ z'$ and $t'$. The cases mentioned down below correspond to the ones such that $B \in \mathcal{M}_{5,5}(\Z)$. For all other possible values $\tau, n, x', y', z'$ and $t'$, the Gram matrices of $B$ 
 are non-integers.\\
 
 For ($\tau=12,\ n=3,\ x'=0,\ y'=0,\ z'=0,\ t'=0 )$, $B$ has the following Gram matrix
 \begin{center}
          $\begin{pmatrix} 3&5 &1&1&2\\
          5&13&0&0&4\\
          1&0&3&0&0&\\
          1&0&0&3&0&\\
          2&4&0&0 & 2  \end{pmatrix}$, 
     \end{center}
    which has a primitive short root $\pm (2,0,-1,-1,-2)$.  \\
 For ($\tau=15,\ n=3,\ x'=0,\ y'=0,\ z'=0,\ t'=0$), $B$ has the following Gram matrix
  \begin{center}
    $\begin{pmatrix} 
3&5 &1&1&2\\
5&13&0&0&5\\
1&0&3&0&0&\\
1&0&0&3&0&\\
2&5&0&0 & 2  \end{pmatrix}$, 
     \end{center}
    which has primitive roots $\pm (1,0,0,-1,-1)$.\\    
Therefore, by \cite[\S 4 Proposition 1]{MR2390291},  no overlattices exist in these two cases.\\

 Thus, for $\tau=11,\ 12,\ 13,\ 14,\ 15$,  $M_{\tau}$ are saturated.\\

The saturated lattice $M_{\tau}$ is definite positive of rank 5 such that: 
 \begin{center}
    $h^2 \in M_\tau \subset M(X) \subset L$. 
 \end{center}
Let $v = x h^2+ yD+ zP+tP'+uT \in M_{\tau}$. For $x, y, z, t, u \in \Z$, we get
 \begin{center}
     $(v,v) = 3x^2 + 13y^2 + 3z^2 +3t^2+18u^2+ 10x y + 2x z +2xt+12xt+2\tau y u$.
 \end{center}
 We prove that $\C_{M_{11}}$ and $\C_{M_{15}}$ are empty by finding roots. Indeed,  $(5 h^2 - D -2P-2P'- T,5 h^2 - D -2P-2P'- T)=2$ is a short root for $\C_{M_{11}}$ and $( h^2+D- T , h^2 +D-T)=2$ is a short root for $\C_{M_{15}}$;  for the remaining values of $\tau$,
 $M_{\tau}$ has no vectors $v$ such that $(v,v) = 2$. Hence, by \cite[Lemma 2.4]{2019arXiv190501936Y}$, \C_{M_{\tau}} \subset \C$ is irreducible and nonempty for $\tau= 11,\ 12,\ 13$.
 
\end{proof} 

\begin{remark}
    In Section 5 we give an example of cubic with $\tau=12$. Examples with $\tau=13, 14$ are easy to construct in the same way, and this is explained in the ancillary file. 
\end{remark}

\begin{prop} \label{Dp4}
      Let $X$ be a generic cubic fourfold in $\C_{M_{12}}$, $\C_{M_{13}}$ or $\C_{M_{14}}$. Then the full linear system $|2h-T -P -P'|$ defines a fibration in del Pezzo quartics over $\P^2$.
\end{prop}

 \begin{proof}
The three quadrics that define $|2h-T|$ automatically contain also $P$ and $P'$, since $T$ intersects the two planes along cubic curves. This means that $|2h-T -P -P'|= |2h-T|$. Clearly the fibers have degree 4 and it is easy to see that the anticanonical bundle of a fiber $F$ is $\mathcal{O}_F(1)$. In order to show the smoothness of the generic fiber, we can construct explicit examples with at least one smooth fiber (and hence the generic fiber is smooth as well) inside each component of Thm \ref{C18/14'} (see \S \ref{section 5}, or the attached file for general method). This in turn implies that for the generic cubic in the three components, the generic quartic fiber is smooth. We can even compute that the fibers are (generically) smooth intersections of two quadrics in $\P^4$.
 \end{proof}

Let $F\in M(X)$ be the class $4h^2-P-P'-T$ of the fiber, a quartic del Pezzo surface. Let $W \in M(X)$ be a $2-$cycle on $X$ such that
\begin{center}
    $W := a h^2 + b D+c P+dP'+e T$, for $a,b,c,d,e \in \Z$
\end{center}

We have that $(W,F) = -5a + b+(20-\tau)c +d+6e$. We observe that a rational
$1-$section exists always (i.e. $(W,F)=1$) for cubics $\C_{M_{12}}$, $\C_{M_{13}}$ or $\C_{M_{14}}$. Recall that, if a quartic del Pezzo surface has $k-$rational points, then it is unirational over $k$. In our case, the quartic del Pezzo surface is unirational over the function field of $\P^2$ and hence unirational over $\mathbb{C}$. This is not surprising since the quartic del Pezzo fibration is birational to the cubic fourfold which is known to be unirational.

\smallskip

As far as the author knows, only one sufficient condition is known for the rationality of a del Pezzo quartic over an arbitrary field. If a quartic del Pezzo surface contains a line defined over the base field then it is rational (see \cite[Proposition 5.4]{ABB}).

We cannot exclude that a line defined over the base field inside the del Pezzo fibration exists, but we can give some evidence for its non existence.

\begin{remark}

The existence of a line over $\P^2$, contained in the quartic del Pezzo fibration birational to the cubic $X$ is equivalent to saying that there exists an open subset $U\subset\P^2$, and a $\P^1$-bundle $W$ over $U$ contained in the family of quartic surfaces over the same open set $U$. Let us show that, under some hypotheses, this $\P^1$-bundle cannot be contained in the cubic fourfold. Suppose that $W$ is contained in the cubic $X$ (i.e. it does not come from the exceptional divisors of the base locus of the linear system $|2h-T|$) and consider the closure $\mathcal{W}$ of $W$ in $X$. The Chow group $CH^1(X)=\Z=\langle h\rangle$ is infinite cyclic, generated by the 
hyperplane section. Hence the cycle $\mathcal{W}$ is a complete intersection of type $(3,d)$. We claim that, if a $(3,d)$ complete intersection has canonical singularities, then it cannot contain a 2-dimensional family of lines. 
The case $(3,2)$ is classical and known since Fano. If $d\geq 3$, then the canonical bundle of the complete intersection is effective, and if the C.I. contained the family of lines, it would be uniruled, which is absurd. Of course if the complete intersection is highly singular, then the canonical class can be negative and this argument does not work.
    
\end{remark}

\begin{remark} \label{quartic/quintic}
\begin{enumerate}

\item It is not hard to construct cubic fourfolds in $\C_{18} \cap \C_{14}$ containing only one plane. We can choose the same $P$, and a second disjoint plane $P''$ such that $P''\cap D=\emptyset$ (see Section \ref{section 5} for details). %
This means that in this case we consider the locus of smooth cubic fourfolds whose lattice of 2-cycles have a sublattice as follows

\begin{center}
 \bordermatrix{
 & h^2 & D & P& T\cr
h^2 & 3 & 5 &1& 6\cr
D & 5 & 13 & 0& \tau \cr
P & 1 & 0&3&0\cr
T & 6 & \tau &0& 18\cr},\end{center}
where $\tau=(D,T)$. A calculation similar to the previous ones shows that this locus has exactly 6 irreducible components $\C_{M_{\tau}}$ for $\tau \in \{9,10,11,12,13,14\}$. An argument similar to Proposition \ref{Dp4}, shows that these components give rise to fibrations in quintic  surfaces over $\P^2$, given by $F'=4h^2-P-T$, with ample anticanonical bundle. We construct an explicit example in the component $\tau=12$ for which the generic fiber is smooth, hence a quintic del Pezzo surface. Also in this case, one can check that the fibration has a rational section. This implies rationality. In fact, a classical theorem by Enriques states that a del Pezzo quintic defined over an infinite field $k$ is always rational over $k$.

\item Despite many attempts, we were unable to produce an example of a cubic fourfold in the intersection of $\C_{18}$ and $\C_{14}$ not containing any plane (i.e. with $rk(M(X))=3$). It is likely that non-generic cubics of this kind have a different geometric construction.

\end{enumerate}

\end{remark}

\smallskip

\subsection{\texorpdfstring{$\C_{18} \cap \C_{26}$}{C18 \textbackslash cap C26}}\label{subsection 3.3}

Let $X$ be a cubic fourfold in $\C_{18} \cap \C_{26}$;  $M(X)$ has primitive sublattices $K_{18}: =\langle h^2,T \rangle$ and $K_{26}:=\langle h^2,S_{26} \rangle$, such that $h^2$ is the square of the hyperplane class, $T$ is (possibly a degeneration of) an elliptic ruled surface and $S_{26}$ is a surface contained in a general element of $\C_{26}$ (as defined in \S \ref{C826}).
 Hence $X \in \C_{18} \cap \C_{26}$ has a sublattice $M_{\tau}:=\langle h^2, T, S_{26}\rangle \subset M(X)$ with Gram matrix:

\begin{center}
     \bordermatrix{
 & h^2 & T& S_{26} \cr
h^2 & 3 & 6 & 7\cr
T & 6 & 18 & \tau \cr
S_{26} & 7 & \tau & 25\cr
},
\end{center}
for some $\tau \in \Z$ depending on $X$.

\begin{thm} \label{C18/26}
The intersection $\C_{18} \cap \C_{26}$ has exactly seven irreducible components $\C_{M_{\tau}}$ given by the following Gram matrix

\begin{center}
$M_{\tau}:=$ \bordermatrix{
 & h^2 & T & S_{26}\cr
h^2 & 3 & 6 & 7\cr
T & 1 & 18 & \tau \cr
S_{26} & 7 & \tau & 25\cr
} or
$B_{12}:=\begin{pmatrix}
 3 & 6 & 8 \\
 6 & 18 & 18  \\
 8 & 18 & 24\\
 \end{pmatrix}$
 \end{center}
where $\tau \in  \{8,9,10,11,12,13,14\}$ and $B_{12}$ is an overlattice of $M_{12}$. Moreover the corresponding discriminants of $M_{8},..,M_{14}$ and $B_{12}$ are 48, 81, 108, 129, 144, 153, 156 and 36 respectively.

\end{thm}

\begin{proof}

The cubic fourfold $X \in\C_{18} \cap \C_{26}$ has a sublattice $M_{\tau}:=\langle h^2, T, S_{26} \rangle \subset M(X)$ with the Gram matrix:

\begin{center}
 \bordermatrix{
 & h^2 & T & S_{26} \cr
h^2 & 3 & 6 & 7\cr
T & 6 & 18 & \tau \cr
S_{26} & 7 & \tau & 25\cr
}

\end{center}
For some $\tau \in \Z$ depending on $X$.\\
As $d(M_{\tau})=-3(\tau^2-28\tau +144)$, the only values making a positive discriminant are $\tau \in \{7,8,9,10,11,12,13,14,15,16,17,18,19,20,21\} $.\\

  Note that there are isometries between the lattices $M_{7}, M_{8}$, $M_{9}$, $M_{10}$, $M_{11}$, $M_{12}$, $M_{13}$  and $M_{21}$, $M_{20}$, $M_{19}$, $M_{18}$, $M_{17}$, $M_{16}$, $M_{15}$ respectively (see \cite[Remark 7.7]{yang2021lattice}). Thus, we may consider only $M_{\tau}$ for $\tau=7,\ 8,\ 9,\ 10,\ 11,\ 12,\ 12,\ 13,\ 14$ of discriminant $d(M_\tau)$ respectively 9, 48, 81, 108, 129, 144, 153, 156.\\
  For these values of $\tau$, no overlattices exist for $M_{\tau}$ except for $(\tau=12;n=2; x'=1; y'=1)$, an overlattice $B_{12}$ exists given by  \begin{center}
 $B_{12}:=\begin{pmatrix}
 3 & 6 & 8 \\
 6 & 18 & 18  \\
 8 & 18 & 24\\
 \end{pmatrix}$

\end{center}
 of discriminant 36.\\
  For these values of $\tau$, $M_{\tau}$ (or $B_{12}$) is a positive definite saturated sublattice of rank 3: 
 \begin{center}
    $h^2 \in M_\tau \subset M(X) \subset L$ 
 \end{center}
 Furthermore, let $v = x h^2+ y T+ z S_{26} \in M_{\tau}$ $x, y, z \in \Z$, we get
 \begin{center}
     $(v,v) = 3x^2 + 18y^2 + 25z^2 +12x y+ 14x z + 2\tau y z$. 
 \end{center} 
$M_{7}$ has roots $\pm (5h^2+T+S_{26})$. For the remaining values of $\tau$, there exists no $v \in M_{\tau}$ such that $(v,v) = 2$.
  Then for these values of $\tau$, $\C_{M_{\tau}} \subset \C$ is irreducible nonempty and has codimension 2 (see \cite[Lemma 2.4]{2019arXiv190501936Y}).
 \end{proof}

Let us assume that the class $4h^2-T \in M(X)$ gives a two-dimensional linear system of sextic del Pezzo surfaces. Under this assumption, we now check the existence of 1,2 and 3-sectional cycles in $M(X)$ for the corresponding del Pezzo fibration. 

For $W_{a,b,c}$ a cycle in $M(X)$ such that
\begin{center}
    $W_{a,b,c} = a h^2 + b T +c S_{26}$, for $a,b,c \in \Z,$
\end{center}
 we have that $(W_{a,b,c},F) = 6a + 6b +(28-\tau)c$.

\smallskip
The following table resumes the situation, according to the different values of $\tau$.

\medskip

\begin{center}
\begin{tabular}{|c|c|c|c|}

\hline
$\tau $ &$(W_{a,b,c},F)=1$ & $(W_{a,b,c},F)=2$ & $ (W_{a,b,c},F)=3$\\
 \hline
8&$\emptyset$ &$W_{0,7,-2}$& $\emptyset$\\\hline 
9&$W_{0,16,-5}$ &$W_{0,13,-4}$& $W_{0,10,-3}$\\\hline
10 &$\emptyset$&$\emptyset$& $\emptyset$\\\hline
11&$W_{0,3,-1}$&$W_{0,6,-2}$&$W_{0,9,-3}$\\\hline
12&$\emptyset$&$W_{0,3,-1}$&$\emptyset$\\\hline
13&$\emptyset$&$\emptyset$&$W_{0,3,-1}$\\\hline
14&$\emptyset$& $W_{0,5,-2}$& $\emptyset$\\\hline

\end{tabular}
\end{center}

\medskip
 
All the observations from Rem. \ref{remarks} hold true also in this case. The components with $\tau=10,\ 13$ parametrize rational cubic fourfolds in $\C_{18}$ such that the associated del Pezzo fibration has no rational section. If the cubics are good, the Brauer class $\beta_2$ is nontrivial.

 \subsection{\texorpdfstring{$\C_{18} \cap \C_{38}$}{C18 \textbackslash cap C38}}\label{subsection 3.4}
 
Let $X$ be a cubic fourfold in $\C_{18} \cap \C_{38}$; $M(X)$ has primitive sublattices $K_{18}: =\langle h^2, T \rangle$ and $K_{38}:=\langle h^2, S_{38} \rangle$, such that $T$ is (a degeneration of) an elliptic ruled surface and $S_{38}$ is as defined in \S \ref{subsection 4.3}.\\
Hence $X \in \C_{18} \cap \C_{38}$ has a sublattice $M_{\tau}:=\langle h^2, T, S_{38}\rangle \subset M(X)$ with Gram matrix:

\begin{center}
     \bordermatrix{
 & h^2 & T& S_{38} \cr
h^2 & 3 & 6 & 10\cr
T & 6 & 18 & \tau \cr
S_{38} & 10 & \tau & 46\cr
},
\end{center}
for some $\tau \in \Z$ depending on $X$.

\begin{thm} \label{C18/38}
The irreducible components of $\C_{18} \cap \C_{38}$ are the subvarieties of codimension two $\C_{M_{\tau}}$  given by rank 3 lattices represented by 

  \begin{center}
        $M_{\tau}:=  \begin{pmatrix} 3&6& 10\\6&18&\tau\\10&\tau & 46 \end{pmatrix}$  
     or 
       $B_{16}:=  \begin{pmatrix} 3&6& 8\\6&18&17\\8&17 & 24 \end{pmatrix}$ 
    or
       $B_{20}:=  \begin{pmatrix} 3&6& 8\\6&18&19\\8&19 & 26 \end{pmatrix}$ 
        
     \end{center}
     where  $12 \leq  \tau \leq 20$, $B_{16}$ and $B_{20}$ are overlattices of $M_{16}$ and $M_{20}$ respectively. Moreover the associated discriminants $d(M_{12}),..,d(M_{20}), d(B_{16})$ and $d(B_{20})$ are 36, 81, 120, 153, 180, 201, 216, 225, 228, 45, and 57 respectively.
\end{thm}

\begin{proof}
As $d(M_{\tau})=-3(\tau^2-40\tau +324)$, the only values of $\tau$ making a positive discriminant are $\tau \in \{12,13,...,28\} $.   Note that there are isometries between the lattices $M_{12}$, $M_{13}$, $M_{14}$, $M_{15}$, $M_{16}$, $M_{17}$, $M_{18}$, $M_{19}$ and $M_{28}$, $M_{27}$, $M_{26}$, $M_{25}$, $M_{24}$, $M_{23}$, $M_{22}$, $M_{21}$ respectively (see \cite[Remark 7.7]{yang2021lattice}). Thus, we may consider only $M_{\tau}$ for $\tau= 12,\ 13,\ 14,\ 15,\ 16,\ 17,\ 18,\ 19,\ 20$. The corresponding discriminants $d(M_\tau)$ are 36, 81, 120, 153, 180, 201, 216, 225, 228.\\
Overlattices exist only for $\tau=16$ and $20$ such that :  \begin{center}
    $B_{16}:=  \begin{pmatrix} 3&6& 8\\6&18&17\\8&17 & 24 \end{pmatrix}$  
        and 
        $B_{20}:=  \begin{pmatrix} 3&6& 6\\8&18&19\\8&19 & 26 \end{pmatrix}$.
        \end{center}
For each of these values of $\tau$, $M_{\tau}$ (as well as $B_{16}$ or $B_{20}$) is a positive definite saturated sublattice of rank 3 containing $h^2$. 

Let $v = x h^2+ y T+ z S_{38} \in M_{\tau}$, for $x, y, z \in \Z$, we get
 \begin{center}
     $(v,v)= 3x^2 + 18y^2 + 46z^2 +12xy+ 20x z + 2\tau y z$.
 \end{center}
There exists no $v \in M_{\tau}$ such that $(v,v) = 2$; thus $\C_{M_{\tau}} \subset \C$ is a nonempty irreducible component and has codimension 2 (see \cite[Thm 5.2]{yang2021lattice}). 

\smallskip
 
\end{proof}
We now look for 1,2 and 3-sectional cycles inside cubics in the irreducible components of $\C_{18}\cap\C_{38}$, assuming that the linear system $|2h-T|$ defines a sextic del Pezzo fibration on $\P^2$.\\ For $W_{a,b,c}$ a cycle in $M(X)$ such that    $W_{a,b,c} = a h^2 + b T+c S_{38}$, for $a,b,c \in \Z$ and $F= 4h^2 - T\in M(X)$, we have that $$(W_{a,b,c},F) = 6a + 6b +(40-\tau)c.$$

\medskip

\begin{center}

\begin{tabular}{|c|c|c|c|}
\hline
$\tau$ &$ (W_{a,b,c},F)=1$ & $ (W_{a,b,c},F)=2$ & $ (W_{a,b,c},F)=3$\\ \hline
12 &$\emptyset$&$W_{0,5,-1}$& $\emptyset$\\\hline
13 &$\emptyset$& $\emptyset$&$W_{0,5,-1}$\\\hline 
14 &$\emptyset$& $W_{0,9,-2}$&$\emptyset$\\\hline 
15 &$W_{0,21,-5}$ &$W_{0,17,-4}$&$W_{0,13,-3}$\\\hline
16 &$\emptyset$& $\emptyset$&$\emptyset$\\
\hline
17&$W_{0,4,-1}$&$W_{0,8,-2}$&$W_{0,12,-3}$\\
\hline
18&$\emptyset$&$W_{0,4,-1}$&$\emptyset$\\\hline
19&$\emptyset$&$\emptyset$&$W_{0,4,-1}$\\\hline
20&$\emptyset$& $W_{0,7,-2}$& $\emptyset$\\\hline

\end{tabular}
\end{center}

Cubic fourfolds from the components with $\tau=13,\ 16,\ 19$ are examples of rational cubics in $\C_{18}$ s.t. the associated del Pezzo fibration has no rational section. If the cubics are good, then the Brauer class $\beta_2$ is non-trivial.

\section{Explicit examples}\label{section 5}

In this section, we shall give explicit examples of  rational cubic fourfolds which are fibered in quadric or del Pezzo surfaces. All our computations have been done using Macaulay2 \cite{M2}. We work over the finite field $\F_{3331}$ but our equations hold over fields of characteristic zero. Hence we will set $\P^5: =Proj(\F_{3331}[x_0,...,x_5])$ and $\P^2:=Proj(\F_{3331}[t_0,...,t_2])$.
\subsection{Cubic fourfold in $\C_{8}$} \label{subsection 5.2}
We provide first an example of rational cubic fourfold in $\C_{8}\cap\C_{38}$ containing a \textit{good} plane with nontrivial \textit{Brauer class} using \S \ref{subsection 4.3}. This is equivalent to the fact that the quadric fibration does not have a rational section.
Let $S_{38}$ be the smooth surface of degree 10 contained in a general element of $\C_{38}$ as defined in \S \ref{subsection 4.3}. It is given by the image of a plane via the linear system of curves of degree 10 with 10 general triple points (see \cite{MR3934590}).\\
Let $P$ be the plane whose ideal is generated by $P_1$, $P_2$ and $P_3$ as follows:
\begin{eqnarray*}
P_1&=&x_2+884x_3+1526x_4+99x_5,\\ 
P_2&=&x_1+363x_3+1053x_4+605x_5,\\
P_3&=&x_0+229x_3+1382x_4+1193x_5.
\end{eqnarray*}
One can compute that $P$ and $S_{38}$ intersect transversally in 2 points. They are contained in the cubic fourfold $Y$ cut out by \\
\begin{equation*}
\begin{aligned}
 C:= &\
x_0^3-559x_0^2x_1-647x_0x_1^2+501x_1^3+1640x_0^2x_2-878x_0x_1x_2-417x_1^2x_2+\\&1333x_0x_2^2-289x_1x_2^2+472x_2^3+103x_0^2x_3+792x_0x_1x_3+183x_1^2x_3-\\&1078x_0x_2x_3-514x_1x_2x_3+1030x_2^2x_3+886x_0x_3^2-727x_1x_3^2-1509x_2x_3^2+\\&609x_3^3-1146x_0^2x_4+1639x_0x_1x_4-397x_1^2x_4+744x_0x_2x_4-1035x_1x_2x_4+\\&174x_2^2x_4+3x_0x_3x_4-153x_1x_3x_4-239x_2x_3x_4+907x_3^2x_4-771x_0x_4^2-\\&1025x_1x_4^2+876x_2x_4^2+633x_3x_4^2-844x_4^3-505x_0^2x_5-889x_0x_1x_5-30x_1^2x_5+\\&822x_0x_2x_5-30x_1x_2x_5-10x_2^2x_5+159x_0x_3x_5+744x_1x_3x_5-851x_2x_3x_5+\\&1187x_3^2x_5+1473x_0x_4x_5-1372x_1x_4x_5-1106x_2x_4x_5+566x_3x_4x_5+957x_4^2x_5\\&+440x_0x_5^2-714x_1x_5^2+278x_2x_5^2-957x_3x_5^2
\end{aligned}
\end{equation*}\\
$Y$ is a smooth cubic hypersurface in $\P^5$ contained in the intersection $\C_{8}\cap\C_{38}$. Let $\tilde{Y}$ be the blow-up of $Y$ along $P$. The discriminant divisor $E \in \P^2$ of the map $q \colon \tilde{Y} \to \P^2 $ is a smooth sextic curve defined as follows:\\
\begin{equation*}
\begin{aligned}
E:&\ t_0^6+1046t_0^5t_1-804t_0^4t_1^2-887t_0^3t_1^3+1253t_0^2t_1^4+58t_0t_1^5+1280t_1^6-164t_0^5t_2+\\
&786t_0^4t_1t_2-674t_0^3t_1^2t_2+960t_0^2t_1^3t_2-574t_0t_1^4t_2-524t_1^5t_2 +783t_0^4t_2^2+\\&1353t_0^3t_1t_2^2-1607t_0^2t_1^2t_2^2-686t_0t_1^3t_2^2+491t_1^4t_2^2-379t_0^3t_2^3+706t_0^2t_1t_2^3-\\&602t_0t_1^2t_2^3-159t_1^3t_2^3+784t_0^2t_2^4+824t_0t_1t_2^4-854t_1^2t_2^4+448t_0t_2^5+1062t_1t_2^5\\&+760t_2^6=0
\end{aligned}
\end{equation*}\\
Hence, $Y$ is a smooth rational cubic fourfold containing a  \textit{good} plane $P$ inside the irreducible component of $\C_{8} \cap \C_{38}$ indexed by $\tau = (P,S_{38}) = 2$. By Corollary \ref{cor 4.6}, the quadric surface bundle associated to cubic fourfolds in this component has no rational section.
\subsection{Cubic fourfold in $\C_{18}$} \label{subsection 5.1}
We exhibit now examples of rational cubic fourfolds $X \in \C_{18} \cap \C_{14}$. The associated del Pezzo fibrations have  rational sections in these cases.
\subsubsection{A cubic fourfold with a quartic del Pezzo fibration}
\medskip
 We will now construct an example of cubic fourfold containing an elliptic ruled surface $T$, the two planes $\Pi_1$ and $\Pi_2$ that we use to construct $T$, and a del Pezzo quintic $D$ that intersects the two planes along smooth conics. From \S \ref{subsection 3.2}, we know the possible values of $\tau= (D,T)$ for cubic fourfolds of this kind. Our example will lie in the component with $\tau=12$, but the same method allows to construct examples in each component (see the ancillary file).\\Let $D$ be a quintic del Pezzo surface in $\P^5$ as defined in \cite[Section 4]{c14}.

To construct $T$, we define three quadrics $Q_1$, $Q_2$ and $Q_3$ as follows:
\begin{eqnarray*}
Q_1&=&x_1^2-884x_0x_2+1331x_1x_2-336x_2^2+538x_1x_3-895x_2x_3-538x_0x_4-\\&&1580x_1x_4+531x_2x_4-644x_4^2-1405x_0x_5-1650x_1x_5+1251x_2x_5-\\&&305x_3x_5-1097x_4x_5-131x_5^2,\\
Q_2&=&x_0x_1-871x_0x_2-512x_1x_2-1526x_2^2-1367x_1x_3-336x_2x_3-213x_0x_4-\\&&304x_2x_4-644x_3x_4+813x_0x_5-169x_1x_5-564x_2x_5+1316x_3x_5-41x_4x_5\\&&+99x_5^2,\\
Q_3&=&x_0^2-1020x_0x_2-102x_1x_2-948x_2^2-1580x_0x_3-1432x_1x_3+1471x_2x_3-\\&&644x_3^2+1432x_0x_4-500x_2x_4+1256x_0x_5-1559x_1x_5-1110x_2x_5\\&&+673x_3x_5-932x_4x_5+939x_5^2.
\end{eqnarray*}
Each of these quadrics contains the two disjoint planes:
\begin{center}
  $\Pi_1=\{x_2-1188x_5=x_1-1188x_4=x_0-1188x_3=0\}$ and $
\Pi_2=\{x_2-392x_5=x_1-392x_4=x_0-392x_3=0\}$
\end{center}
$T$ is obtained by saturating the ideal generated by $Q_1$, $Q_2$ and $Q_3$ with respect to the defining ideals of the planes $\Pi_1$ and $\Pi_2$. \\
The surfaces $T$ and $D$ are contained in the cubic fourfold $X$ cut out by
\begin{equation*}
\begin{aligned}
C':= &\ x_0^2x_1+1200x_0x_1^2+1052x_0x_1x_2-1200x_1^2x_2-1053x_1x_2^2+1516x_0x_1x_3\\&+1133x_1^2x_3+1218x_0x_2x_3+696x_1x_2x_3+333x_2^2x_3-1459x_1x_3^2-\\& 1145x_2x_3^2+235x_0^2x_4+1537x_0x_1x_4-1420x_0x_2x_4 +1093x_1x_2x_4 +\\& 688x_2^2x_4+815x_0x_3x_4+1301x_1x_3x_4-161x_2x_3x_4-1309x_0x_4^2-398x_2x_4^2\\&-1219x_0^2x_5-466x_0x_1x_5 +768x_1^2x_5-1386x_0x_2x_5-1395x_1x_2x_5-\\&606x_0x_3x_5-646x_1x_3x_5+1442x_2x_3x_5+644x_3^2x_5 -1407x_0x_4x_5-\\&255x_1x_4x_5+1394x_2x_4x_5-636x_3x_4x_5-8x_4^2x_5+129x_0x_5^2-1026x_1x_5^2-\\&1392x_3x_5^2+1392x_4x_5^2.
\end{aligned}
\end{equation*}
$X$ is a smooth irreducible subscheme of $\P^5$ of dimension 4 and degree 3. The surface $T$ (resp. $D$) cuts a smooth cubic curve (resp. smooth conic) out of $\Pi_1$ and $\Pi_2$. This implies directly that $X$ contains the two planes, and that $(T,D)\geq 12$, since the intersection of the plane cubic and conic gives 6 points. A quick Macaulay2 calculation shows that $T$ and $D$ indeed intersect transversally in 12 points and this in turn means that our example belongs to the discriminant 108 component (see Thm. \ref{C18/14'}).\\
The linear system $|2h - T|=\P^2$ contains also $\Pi_1$ and $\Pi_2$ in its base locus, since $T\cap \Pi_i$ is a cubic curve for $i=1,2$, hence the fibers have degree 4. Let $\tilde{X}$ be the blow-up of $X$ along $T$. The discriminant locus of the map  $\pi \colon \tilde{X} \to \P^2 $ given by $|2h - T|$ is a reducible curve of degree 12 with two irreducible components. The generic fiber is a del Pezzo surface of degree 4 (a smooth intersection of two quadrics).

\subsubsection{Cubic fourfold with a quintic del Pezzo fibration}

To construct $T$, we define three quadrics $Q'_1$, $Q'_2$ and $Q'_3$ as follows:
\begin{eqnarray*}
Q_1'&=&x_1^2+1065x_0x_2+175x_1x_2-70x_2^2+1044x_0x_3-1139x_1x_3+1091x_2x_3-\\&&1140x_3^2-1024x_0x_4+1571x_1x_4+1184x_2x_4+1443x_3x_4+12x_4^2-\\&&1027x_0x_5+802x_1x_5-1655x_2x_5-339x_3x_5+768x_4x_5+801x_5^2,\\
Q_2'&=&x_0x_1+840x_0x_2-1358x_1x_2+77x_2^2-37x_0x_3+653x_1x_3+726x_2x_3+\\&&653x_3^2-1360x_0x_4-618x_1x_4-759x_2x_4+1504x_3x_4+1364x_4^2-\\&&161x_0x_5+1577x_1x_5 +1014x_2x_5-1101x_3x_5-1415x_4x_5+1577x_5^2,\\
Q_3'&=&x_0^2 +1415x_0x_2-34x_1x_2+765x_2^2-615x_0x_3-1200x_1x_3+1008x_2x_3-\\&&1200x_3^2+1194x_0x_4-1191x_1x_4-222x_2x_4+466x_3x_4-767x_4^2-\\&&901x_0x_5-13x_1x_5-1058x_2x_5-1213x_3x_5-1314x_4x_5-13x_5^2.
\end{eqnarray*}
Each of these quadrics contains the two disjoint planes:
\begin{center}
  $\Pi_1=\{x_2-1188x_5=x_1-1188x_4=x_0-1188x_3=0\}$ and $
  O=\{x_2 + x_4= x_1 + x_3 + x_5= x_0 - x_4=0\}$
\end{center}
$T$ is obtained by saturating the ideal generated by $Q_1$, $Q_2$ and $Q_3$ with respect to the defining ideals of the planes $\Pi_1$ and $O$. While $\Pi_1$ is the same plane as in the preceding example (hence cutting out a cubic and a conic on $T$ and $D$), $O$ is a different plane, that still intersects $T$ along a smooth cubic but has empty intersection with $D$.
The surfaces $T$ and $D$ are contained in the cubic fourfold $X$ of equation
\begin{equation*}
\begin{aligned}
C'':=&\ x_0^2x_1-634x_0x_1^2-983x_0x_1x_2+634x_1^2x_2+982x_1x_2^2-769x_0x_1x_3-1007x_1^2x_3\\&+730x_0x_2x_3-208x_1x_2x_3-1249x_2^2x_3-528x_1x_3^2-911x_2x_3^2-572x_0^2x_4+\\&371x_0x_1x_4-133x_0x_2x_4+59x_1x_2x_4 -1255x_2^2x_4-913x_0x_3x_4+560x_1x_3x_4+\\&1154x_2x_3x_4+1092x_0x_4^2+77x_2x_4^2-731x_0^2x_5+137x_0x_1x_5-57x_1^2x_5-\\&1100x_0x_2x_5-878x_1x_2x_5-1079x_0x_3x_5-1578x_1x_3x_5+1172x_2x_3x_5+\\&1441x_3^2x_5+127x_0x_4x_5+966x_1x_4x_5+1009x_2x_4x_5+238x_3x_4x_5+1652x_4^2x_5\\&+1008x_0x_5^2+1175x_1x_5^2+1033x_3x_5^2-1033x_4x_5^2.
\end{aligned}
\end{equation*}
$X$ is a smooth irreducible subscheme of $\P^5$ of dimension 4 and degree 3 that contains $\Pi_1$ (by the same argument as the preceding example), $T$ and $D$. A Macaulay2 computation shows that $(D,T)=12$, and 6 of these points are contained in $\Pi_1$. The linear system $|2h - T|=\P^2$  contains also $\Pi_1$ in its base locus, since $T\cap \Pi_1$ is a cubic curve, hence the degree of the fibers of the induced map $X\to \P^2$ is five. The discriminant locus of the map  $\tilde{X} \to \P^2 $ is a reducible curve of degree 12 with two irreducible components: a singular curve of degree 6 and a smooth curve of degree 6.
The generic fiber $F$ is smooth, has degree 5, and it is not hard to show that $\mathcal{O}_F(1)$ is the anticanonical bundle. Hence the generic fiber is a del Pezzo quintic surface. \\

%\paragraph{\textbf{Data Availability}}Data sharing not applicable to this article as no datasets were generated or analysed during the current study.
%\bibliography{bib_tocho}
%\bibliographystyle{abbrv}

\bibliography{bib_tocho}
\bibliographystyle{abbrv}

\end{document}